\documentclass[english]{amsart}
\usepackage[T1]{fontenc}
\usepackage[latin9]{inputenc}
\usepackage{amsmath}
\usepackage{amssymb}
\usepackage{amsmath}
\usepackage{amstext}
\usepackage{amsfonts}
\usepackage{amscd}
\usepackage{amsbsy}
\usepackage{amsthm}
\usepackage{ifthen}
\usepackage{enumerate}
\usepackage{latexsym}
\usepackage{babel}
\usepackage{xy}
\usepackage{graphics}
\xyoption{all}
\usepackage{hyperref}
\usepackage{bbm}
\usepackage{microtype}

\makeatletter


\newcounter{dummy} \numberwithin{dummy}{section}
\newtheorem*{theorem*}{Theorem}

\newtheorem{theorem}[dummy]{Theorem}
\newtheorem{prop}[dummy]{Proposition}
\newtheorem{lemma}[dummy]{Lemma}

\newtheorem{corollary}[dummy]{Corollary}

\newtheorem{question}{Question}

\theoremstyle{definition}

\newtheorem{definition}[dummy]{Definition}
\newtheorem{example}[dummy]{Example}
\newtheorem{con}[dummy]{Construction}
\newtheorem{remark}[dummy]{Remark}

\newcommand{\spec}{\mathop{\mathbf{Spec}}\nolimits}

\newcommand{\tor}{\mathop{\rm Tor}\nolimits}

\newcommand{\id}{\mathop{\rm \textbf{id}}\nolimits}

\newcommand{\Md}{\mathop{\rm Mod}\nolimits}

\newcommand{\ann}{\mathop{\rm ann}\nolimits}

\newcommand{\thick}{\mathop{\rm Thick}\nolimits}

\newcommand{\da}{\dagger}

\newcommand{\ass}{\mathop{\rm ass}\nolimits}

\newcommand{\perf}{\mathop{\rm Perf}\nolimits}

\newcommand{\cone}{\mathop{\rm cone}\nolimits}

\newcommand{\spc}{\mathop{\mathbf{Spc}}\nolimits}
\newcommand{\hspc}{\mathop{\mathbf{HSpc}}\nolimits}
\newcommand{\supp}{\mathop{\mathbf{supp}}\nolimits}
\newcommand{\smsupp}{\mathop{\mathbf{supp}_{smash}}\nolimits}
\newcommand{\Supp}{\mathop{\mathbf{Supp}}\nolimits}
\newcommand{\loc}{\mathop{\rm loc}\nolimits}
\newcommand{\lspc}{\mathop{\mathbf{Lspc}}\nolimits}
\newcommand{\bigmt}{\bigwedge}
\newcommand{\bigjn}{\bigvee}

\newcommand{\sT}{\mathop{\rm T}\nolimits}
\newcommand{\wass}{\mathop{\widetilde{\rm Ass}}\nolimits}
\newcommand{\aspec}[1]{{}^{a}\!{#1}}
\newcommand{\FRAME}{\mathop{\rm \textit{\textbf{Frm}}}\nolimits}
\newcommand{\TOP}{\mathop{\rm \textit{\textbf{Top}}}\nolimits}
\newcommand{\SPAT}{\mathop{\rm \textit{\textbf{Spt}}}\nolimits}
\newcommand{\SOB}{\mathop{\rm \textit{\textbf{Sob}}}\nolimits}

\newcommand{\Hck}[3]{{\rm \check{H}}^{#1}_{#2}\!\left(#3\right)}
\newcommand{\HH}[3]{{\rm H}^{#1}_{#2}\!\left(#3\right)}
\newcommand{\KK}[2]{{\rm K}_{#1}\!\left(#2\right)}
\newcommand{\DD}[1]{{\rm D}\!\left(#1\right)}

\newcommand{\Ze}[1]{\mathbf{Z}\!\left(#1\right)}
\newcommand{\wco}[1]{{#1}^{\mathbf{c}}}
\newcommand{\co}[1]{{#1}^{{\rm c}}}
\newcommand{\Neu}[1]{\mathbb{N}\!\left(#1\right)}
\newcommand{\Ab}[1]{\mathbb{D}\!\left(#1\right)}
\newcommand{\skula}[1]{{\rm \bf skula}\!\left(#1\right)}
\newcommand{\Fr}[1]{\mathbb{F}\!\left(#1\right)}
\newcommand{\TT}[1]{\mathbb{T}\!\left(#1\right)}
\newcommand{\hm}[3]{{\rm Hom}_{#1}\left(#2,#3\right)}
\newcommand{\im}[1]{{\rm Im}\!\left(#1\right)}
\newcommand{\Ac}[1]{\mathcal{A}\!\left(#1\right)}
\newcommand{\HKinf}[3]{{\rm H}^{#1}{\rm K}_\infty\!\left(#2\colon\!#3\right)}
\newcommand{\HKinfbig}[3]{{\rm H}^{#1}{\rm K}_\infty\!\big(#2\colon\!#3\big)}
\newcommand{\Kinf}[2]{{\rm K}_\infty\!\left(#1\colon\!#2\right)}

\newcommand{\HHbig}[3]{{\rm H}^{#1}_{#2}\!\big(#3\big)}

\newcommand{\sbe}{\subseteq}

\newcommand{\sbne}{\subsetneq}

\newcommand{\tns}{\otimes}

\newcommand{\xto}{\xrightarrow}

\newcommand{\mt}{\land}
\newcommand{\jn}{\lor}
\newcommand{\op}{\oplus}

\newcommand{\lns}{\otimes^\mathbf{L}}

\newcommand{\ZZ}{\mathbb{Z}}
\newcommand{\CC}{\mathbb{C}}
\newcommand{\NN}{\mathbb{N}}

\providecommand{\bf}{\mathbbm{f}}

\newcommand{\bp}{\mathbbm{p}}
\newcommand{\bq}{\mathbbm{q}}

\newcommand{\bx}{\mathbbm{x}}
\newcommand{\by}{\mathbbm{y}}

\newcommand{\bB}{\mathbb{B}}

\newcommand{\bF}{\mathbb{F}}

\newcommand{\bL}{\mathbb{L}}
\newcommand{\bM}{\mathbb{M}}
\newcommand{\bN}{\mathbb{N}}

\newcommand{\bS}{\mathbb{S}}
\newcommand{\bT}{\mathbb{T}}

\newcommand{\bX}{\mathbb{X}}
\newcommand{\bY}{\mathbb{Y}}
\newcommand{\bZ}{\mathbb{Z}}

\newcommand{\cL}{\mathcal{L}}
\newcommand{\T}{\mathcal{T}}

\newcommand{\cD}{\mathcal{D}}

\newcommand{\cI}{\mathcal{I}}
\newcommand{\cJ}{\mathcal{J}}

\newcommand{\cP}{\mathcal{P}}

\newcommand{\cS}{\mathcal{S}}
\newcommand{\cT}{\mathcal{T}}

\newcommand{\cX}{\mathcal{X}}
\newcommand{\cY}{\mathcal{Y}}

\newcommand{\mm}{\mathfrak{m}}

\newcommand{\p}{\mathfrak{p}}
\newcommand{\q}{\mathfrak{q}}

\newcommand{\gp}{\mathfrak{p}}

\newcommand{\gX}{\mathfrak{X}}

\newcommand{\wc}{\mathbf{c}}

\newcommand{\wf}{\mathbf{f}}
\newcommand{\wg}{\mathbf{g}}

\newcommand{\wl}{\boldsymbol{\ell}}

\newcommand{\ws}{\mathbf{s}}

\newcommand{\wC}{\mathbf{C}}
\newcommand{\wD}{\mathbf{D}}

\newcommand{\wI}{\mathbf{I}}

\newcommand{\wR}{\mathbf{R}}
\newcommand{\wS}{\mathbf{S}}

\newcommand{\wU}{\mathbf{U}}
\newcommand{\wV}{\mathbf{V}}
\newcommand{\wW}{\mathbf{W}}
\newcommand{\wX}{\mathbf{X}}
\newcommand{\wY}{\mathbf{Y}}

\newcommand{\al}{\alpha}
\newcommand{\be}{\beta}
\newcommand{\ga}{\gamma}
\newcommand{\Ga}{\Gamma}

\newcommand{\ld}{\lambda}

\newcommand{\sm}{\sigma}
\newcommand{\ph}{\varphi}
\newcommand{\Sm}{\Sigma}

\title[Support in tensor triangulated categories]{Support and vanishing for non-Noetherian \\ rings and tensor triangulated categories}
\author{William T.\ Sanders}

\begin{document}

\maketitle

\begin{abstract}
We define and characterise small support for complexes over non-Noetherian rings  and in this context prove a vanishing theorem for modules.  Our definition of  support makes sense for any rigidly compactly generated tensor triangulated category.  Working in this generality, we establish basic properties of support and investigate when it detects vanishing.  We use pointless topology to relate support, the topology of the Balmer spectrum, and the structure of the idempotent Bousfield lattice.  
\end{abstract}



\section{Introduction}

In this article, we propose the following definition. 

\begin{definition}\label{introdef}

For an arbitrary commutative ring $R$ and a complex $M$ of $R$-modules, we say $\p\in \spec R$ is in $\supp M$ if for every finite subset  ${\underline{x}=x_1,\dots,x_n\in \p}$,  the cohomology
\[\HKinf{i}{\underline{x}}{M_\p}=\HHbig{i}{}{ (R\to R_{x_1})\tns\cdots\tns(R\to R_{x_n})\tns M_\p}\ne 0\]
does not vanish for some $i\in \ZZ$.

\end{definition}

We justify this definition with the following result, which is Theorem \ref{main6} in the text.  

\begin{theorem}\label{intro1} 

Let $M$ be a complex over a commutative ring $R$.  Suppose that either
\begin{itemize}
\item $H^i(M)=0$ for $i\ll 0$, e.g.\ $M$ is a module
\item or the prime ideals of $R$ satisfy the descending chain condition.  
\end{itemize}
Then $\supp M=\emptyset$ if and only if $M=0$.  
\end{theorem}

 In \cite{Foxby79}, assuming  $R$ is commutative Noetherian ring and $M$ a complex of $R$-modules,   Foxby defined  $\supp M$ as the primes $\p$ such that $M\lns k(\p)\ne 0$ where $k(\p)$ is the residue field $R_\p/\p R_\p$.  In this work, Foxby showed that this support detected vanishing, i.e. $\supp M=\emptyset$ if and only if $M=0$. By  \cite[(2.1) and (4.1)]{FoxbyIyengar03}, our definitions of support coincide in the Noetherian case.

Support is a powerful tool for  triangulated categories.  In \cite{Neeman92}, Neeman classified the localising subcategories of   $\DD{R}$ in terms of support.  In \cite{BensonIyengarKrause08},  Benson, Iyengar, and Krause  construct a support theory in a compactly generated triangulated category acted upon by a Noetherian ring.  In \cite{BensonIyengarKrause11b}  they used this support theory to classify the localising subcategories of the stable category  of modular group representations.  

For non-Noetherian rings, Foxby's theory of supports breaks down: there are modules $M$ such that $M\lns k(\p)=0$ for all $\p\in \spec R$; see Example \ref{Jan}.  Moreover,  Neeman's classification of localising subcategories fails spectacularly for non-Noetherian rings; see  \cite{Neeman00} or more generally \cite{DwyerPalmieri08}.

 In \cite{BalmerFavi11} Balmer and Favi define support in certain rigidly compactly generated tensor triangulated categories.  Their support takes values in the Balmer spectrum of the compact objects, and their definition is valid whenever this space is topologically Noetherian.  Greg Stevenson studied this support in \cite{Stevenson13b} and applied these results to the derived category of an absolutely flat ring in \cite{Stevenson14} and \cite{Stevenson17}. These results suggest that even though Neeman's classification fails, support detects some semblance of order in the localising subcategories.

%
%
%
%
%
%
Following the example of Balmer, Favi, and Stevenson, we study support in rigidly compactly generated tensor triangulated categories.  We introduce the following definition.  See Section \ref{TTpre} for relevant definitions.  

\begin{definition}
Let $\T$ be a rigidly compactly generated tensor triangulated category.  For a prime in the Balmer spectrum $\p\in \spc \T^c$ and an object $T\in \T$, we say $\p\in \supp T$ if for every Thomason subset $\wV\sbe \spc \T$ with $\p\in \wV$, 
\[\Ga_\wV T_\p\ne 0.\] 
We call $\T$ \textit{supportive} if support detects vanishing, i.e.\ $\supp T=\emptyset$ if and only if $T=0$.  
\end{definition}

When $\T=D(R)$, this definition specialises to Definition \ref{introdef}; see Lemma \ref{ringsup}.  When $\spc \T^c$ is topologically Noetherian, our definition coincides with Balmer and Favi's in  \cite{BalmerFavi11}.  

Unfortunately, we do not know if $\T$ is always supportive.  But when it is, our support exhibits strong properties and  behaves similarly to the support developed by Benson, Iyengar, and Krause in \cite{BensonIyengarKrause08}.   See  Theorems  \ref{main1} and \ref{main2}.   Moreover, our support is the only reasonable support function taking values in $\spec R$; see Theorem \ref{main3}.
We  summarise these results below.

\begin{theorem}\label{intro2} 
Let $T\in \T$, and consider the following properties of a function 
\[\ws\colon\T\to \{\mbox{\rm Subsets of } \spc \T^c\}.\]
\begin{enumerate}
\item\label{Vanishing} \textbf{Vanishing}: $\ws(T)=\emptyset$ if and only if $T=0$. 
\item\label{Local} \textbf{Local}: If $\wV\sbe\spc \T^c$ is Thomason and $T$ is $\wV$-local, then $\ws(T)\cap \wV=\emptyset$.
\item\label{Big Support} \textbf{Big Support}: $\ws(T)\sbe \Supp T$ for any compact object $T\in \T^c$.
\item\label{Orthogonality} \textbf{Orthogonality}: For any $S\in \T$, if there is a Thomason subset $\wV\sbe\spc \T^c$  such that $\ws(T)\sbe \wV$ and $\ws(S)\cap \wV=\emptyset$, then 
\[\hm{\T}{T}{S}=0.\]
\item\label{Exactness} \textbf{Exactness}: For any exact triangle $S\to T\to S'\to$ in $\T$,
\[\ws(T)\sbe \ws(S)\cup\ws(S').\]
\item\label{Separation} \textbf{Separation}: For any Thomason subset $\wV\sbe\spc \T^c$, there is an exact triangle
\[T'\to T\to T''\to\]
with $\ws(T')\sbe \wV$ and $\ws(T'')\cap\wV=\emptyset$.
\end{enumerate}
If $\T$ is supportive, then $\supp$ satisfies all of these properties.  Moreover, if a function $\ws$ satisfies all of these properties, then $\ws=\supp$ and $\T$ is supportive.  
%
\end{theorem}

We know of may instances where $\T$ is supportive and none where it is not.    The following result is Corollary \ref{main4cor}, Theorem \ref{main5}, and Theorem \ref{main6}.

\begin{theorem}\label{intro3} 
A rigidly compactly generated tensor triangulated category $\T$ is supportive in the following cases.
 \begin{enumerate}
 \item The idempotent Bousfield lattice of $\T$ is a spatial frame.
 \item The Hochster dual of the Balmer spectrum $\spc \T^c$ is weakly scattered: for every Thomason set $\wU\sbe \spc \T^c$, there is a point $\p\notin \wU$ and a Thomason subset $\wV$ such that
 \[\{\p\}\sbe \wV\cap \wco{\wU}=\downarrow \p.\]
 \item $\T=\DD{R}$ with $R$ a commutative ring satisfying DCC on prime ideals.  
 \end{enumerate}
 \end{theorem}


In Section \ref{Toppre} and Section \ref{TTpre} we discuss preliminary material, including the basics of spectral spaces, pointless topology, the Balmer spectrum, and localisation functors.  In Section \ref{supportstuff} we give the definition of support, discuss its various properties, and prove Theorem \ref{intro2}. In Section \ref{commstuff} we specialise  to the case $\T=\DD{R}$ with $R$ a commutative ring, proving Theorem \ref{intro1}. In Section \ref{framestuff} we relate  support to the Bousfield lattice, and characterise in Theorem \ref{main4} the supportive condition using pointless topology.   In Section \ref{assemblystuff}, we study topological conditions on the Balmer spectrum which imply supportive.   In Section \ref{questionstuff}, we pose several questions.   

We close this section by establishing some conventions.  Triangulated categories and their subcategories will generally be denoted with curly letters such as $\T$ while their objects will be noted with roman capital letters such as $T$.  Rings, modules and complexes will also be denoted in roman capital letters, and their elements in lowercase roman letters.  Topological spaces and their subsets will be denoted with bold capital roman fonts  and their points  with lowercase gothic fonts, e.g.\ $\p\in \wX$.  Continuous functions will be denoted with bold  lowercase letters like $\wf$.  Lattices will be denoted with blackboard bold fonts like $\bX$, and their elements in lower case e.g.\ $\bx\in\bX$.   Greek letters will denote lattice homomorphisms.

\section{Topological preliminaries}\label{Toppre}

\subsection{Point-set topology}\label{Toppre1}

\begin{definition}\label{twiggywiggy}
A space $\wX$ is called \textit{spectral} if 
\begin{enumerate}[(a)]
\item $\wX$ is $\sT_0$
\item every irreducible closed set $\wV\sbe \wX$ has a generic point, i.e.\ there is a $\gp\in\wV$ such that $\wV=\bar{\p}$
\item the quasi-compact open subsets of $\wX$ are a basis
\item the intersection of any two quasi-compact open sets is again quasi-compact open
\item the space $\wX$ is quasi-compact.  
\end{enumerate}
\end{definition}

Spectral spaces were introduced by Hochster, where he gave the following classification.

\begin{theorem}[{\cite{Hochster69}}]

A space $\wX$ is spectral if and only if there exists a commutative ring $R$ such that $\spec R\cong \wX$.  

\end{theorem}

Any spectral apace has another important topology.  

\begin{definition}
Let $\wX$ be a spectral space.  
\begin{enumerate}
\item A subset $\wV\sbe \wX$ is \textit{Thomason} if it is the union of complements of quasi-compact open sets. 
\item The quasi-compact open sets of a spectral space form a closed base, and the topology that they define is called the \textit{Hochster dual} of $\wX$ and denoted $\wX^\da$. The open subsets of $\wX^\da$ are the Thomason subsets of $\wX$.  
\end{enumerate}
\end{definition}

\begin{theorem}[{\cite[Proposition 8]{Hochster69}}]

For any spectral space $\wX$, the Hochster dual $\wX^\da$ is also spectral.  Moreover, there is a natural homeomorphism $\wX\cong{(\wX^{\da})}^\da$.  

\end{theorem}

The points in a spectral space $\wX$ are partially ordered: for $\p,\q\in \wX$ we say that $\p\sbe \q$ if $\q\in \bar{\p}$.  If  $\spec R\cong \wX$ for a commutative ring $R$, this partial order is induced by the inclusion of prime ideals in $R$.  

\begin{example}\label{ZPThomason}

Let $\wX$ be a spectral space and $\p\in \wX$ be a point.  Define
\[\Ze{\p}=\{\q\in \wX\mid \q\nsubseteq \p\}.\]
This is the largest Thomason subset   of $\wX$ not containing $\p$.  Therefore, the closure of $\p$ in the Hochster dual is
\[\downarrow \p=\{\q\in \spec R\mid \q\sbe \p\}.\]

To see that $\Ze{\p}$ is Thomason, assume that $\wX=\spec R$ and $\p\sbe R$ is a prime ideal for a commutative ring $R$.  Then
\[\Ze{\p}=\bigcup_{x\notin \p} \wV(x).\]
The complement of each $\wV(x)$ is quasi-compact,  proving the claim.  
\end{example}

Recall the following definitions from point-set topology.  

\begin{definition}\label{skuladef}

Let $\wX$ be an arbitrary topological space.  
\begin{enumerate}
\item Denoted by $\skula{\wX}$, the \textit{Skula} or \textit{front} topology on $\wX$ is the weakest topology where every open set of $\wX$ is open and closed.  
\item Let $\wf_{\wX}\colon \skula{\wX}\to \wX$ be the set theoretic identity function.  
\item The space $\wX$ satisfies the  separation axiom $\sT_{\frac{1}{2}}$ if for every point $\p\in \wX$ there is an  open set $\wV\sbe \wX$ and a closed set $\wU\sbe \wX$ such that $\{\p\}=\wV\cap\wU$.   

\end{enumerate}
\end{definition}

The Skula topology on $\wX$ is discrete if and only if it is $\sT_{\frac{1}{2}}$; for every point $\p\in \wX$.

\subsection{Pointless topology}\label{prelimframes}

In this section we discuss the basics of pointless topology.  The reader  should refer to \cite{Johnstone82} or  \cite{PicadoPultr12} for further reading.
\begin{definition}\ 
\begin{enumerate}
\item A \textit{complete lattice} is a partially ordered set $\bL$  such that every subset admits both a supremum and an infimum, i.e.\  a join denoted $\jn$ and a meet denoted $\mt$.   In particular every frame has a maximum and minimum denoted by $1_\bL$ and $0_\bL$ respectively, or more often simply 1 and 0.  
\item A \textit{frame} $\bX$ is a complete lattice such that for every $\bx\in \bX$ and $\{\by_i\}\sbe \bF$ satisfies
\[\bx\mt\bigjn_i \by_i=\bigjn_i \bx\mt \by_i.\]
\item A map $\ph\colon\bX\to\bY$ between frames is a \textit{frame homomorphism} if 
	\begin{enumerate}[(a)]
	\item $\ph$ preserves the maximum and minimum, i.e.
	\[\ph(1_\bX)=1_\bY\quad\quad\quad\quad \ph(0_\bX)=0_{\bY}\]
	\item	$\ph$ preserves arbitrary joins, i.e.\ every $\{\bx_i\}\sbe\bX$ satisfies
	\[\ph\left(\bigjn_i\bx_i\right)=\bigjn_i \ph(\bx_i)\]
	\item $\ph$ preserves finite meets, i.e.\ every $\bx,\by\in \bX$ satisfies
	\[\ph(\bx\mt\by)=\ph(\bx)\mt \ph(\by).\]
	\end{enumerate}
\item Let $\FRAME$ be the category of  frames and frame homomorphisms.  
\end{enumerate}
\end{definition}
The following is the critical example  of a frame.
\begin{example}
For a topological space $\wX$, let $\Fr{\wX}$ be the open sets of $\wX$ partially ordered by inclusion.   It is easy to check that $\Fr{\wX}$ is a frame where joins are  unions and finite meets are  intersections.  Given a continuous function $\wg\colon\wX\to \wY$, the induced map 
\[\Fr{\wg}=\wg^{-1}\colon \Fr{\wY}\to \Fr{\wX}\]
is a frame homomorphism. Thus we have defined a contravariant functor 
\[\bF\colon\TOP\to\FRAME.\]  
We say that a frame is \textit{spatial} or \textit{has enough points} if it is isomorphic to a frame  in the image of $\bF$.  
\end{example}

\begin{definition}
Let $\bX$ be a frame.
\begin{enumerate}
\item An element $\bp\in\bX$ is \textit{meet irreducible} or \textit{prime} if for any $\bx,\by\in \bX$ 
\[\bx\mt\by\le \bp\]
implies either $\bx\le \bp$ or $\by\le \bp$.  
\item Let $\spc \bX$ be the set of meet irreducible elements of $\bX$.
\item For any $\bx\in \bX$, let $\wD(\bx)\sbe \spc \bX$ be the meet irreducible $\bp$ such that $\bx\nleq\bp$. 
\end{enumerate}
\end{definition}

\begin{lemma}[{\cite[II.4.1, II.4.3.1]{PicadoPultr12}}]
The set  $\spc \bX$ is a topological space whose open sets are of the form $\wD(\bx)$.  Moreover, we have actually defined a contravariant functor $\spc\colon \FRAME\to\TOP$. 
\end{lemma}

A space is \textit{sober} if it is homeomorphic to $\spc \bX$ for a frame $\bX$.  Sober spaces are ubiquitous.  Indeed, a space is sober if and only if it is $\sT_0$ and every irreducible closed set has a generic point, see \cite[II.1.7]{Johnstone82}. Thus every Hausdorff and  spectral space is sober; see  \cite[II.1.6]{Johnstone82} and Definition \ref{twiggywiggy}.  On the other hand, the maximal ideal spectrum of $\CC[x,y]$ is not sober. 

Let $\SPAT$ and $\SOB$ denote the full subcategories of $\FRAME$ and $\TOP$ respectively consisting of spatial frames and sober spaces.  

\begin{theorem}[{\cite[II.1.7 Corollary and II.1.5 and  II.1.6]{Johnstone82}}]\label{framecats}\hfill
\begin{enumerate}
\item The functors $\bF$ and $\spc$ restrict to an equivalence of categories
\[\SOB\cong \SPAT^{\mbox{op}}.\]
\item For every frame $\bX$, there is a natural frame homomorphism
\[\ld_\bX\colon \bX\to \Fr{\spc \bX}.\]
Moreover, $\bX$ is spatial if and only if $\ld_\bX$ is an isomorphism.  
\item For any space $\wX$ there is a natural continuous function
\[\wl_{\wX}\colon \wX\to \spc \Fr{\wX}.\]
Moreover, $\wX$ is sober if and only if $\wl_\wX$ is an isomorphism.
\end{enumerate} 
\end{theorem}

Thus, sober spaces are completely described by their associated frames, and spatial frames completely describe their associated space.

\section{Tensor triangulated preliminaries}\label{TTpre}

%
%
%
%
%
%

\subsection{The Balmer spectrum}

Let $(\T,\Sm,\tns,1)$ be an essentially small tensor triangulated category.  This means that $\T$  is a triangulated category with shift functor $\Sm$ and a closed symmetric monoidal product $\tns$.  Thus we assume the following: there is a tensor unit $1\in \T$; $S\tns T\cong T\tns S$ for all $S,T\in\T$; and that $\tns$ is exact. Lastly, we assume that the isomorphism classes of $\T$  form a set.

\begin{definition}\label{balmerspec}\  
\begin{enumerate}
\item A subcategory  $\cI\sbe \T$ is a thick tensor ideal if 
	\begin{enumerate}[(a)]
	\item $\cI$ is sub triangulated i.e.\ contains 0 and is closed under mapping cones
	\item $\cI$ is closed under direct summands
	\item for every $S\in \cI$ and $T\in \T$, $S\tns T\in\cI$.  
	\end{enumerate}
\item A thick tensor ideal $\cI\sbe \T$ is radical if $T\tns \cdots \tns  T\in \cI$ implies $T\in \cI$.  
\item A thick tensor ideal $\cP\sbe \T$ is prime if $\cI\tns \cJ\sbe \cP$ implies either $\cI\sbe \cP$ or $\cJ\sbe \cP$ for any thick tensor ideals $\cI,\cJ\sbe \cT$.  Equivalently, prime thick tensor ideals are the meet irreducible elements of the lattice of radical thick tensor ideals.  
\item For an object $T\in \T$,  let $\Supp T$ denote the set of primes  which do not contain $T$.  
\item Let $\spc \T$ be the set of prime thick tensor ideals of $\T$.  We consider the weakest topology such that  $\Supp T$ is closed for all $T\in \T$.  This topological space is called the Balmer Spectrum.
\end{enumerate}
\end{definition}

The prototypical example is taken from considering $\T=\perf(R)$ the perfect complexes over a commutative ring $R$.  In this case, $\spc \T$ is homeomorphic to $\spec R$.  The construction above is given by Balmer in \cite{Balmer05}.  

\begin{theorem}\label{class}\ 

\begin{enumerate}
\item\label{taaaa} The Balmer spectrum $\spc \T$ is always spectral.  
\item\label{taaa} The support function $\Supp$ gives a bijection between the radical tensor ideals of $\T$ and the Thomason subsets of $\spc \T$. 
\end{enumerate}

\end{theorem}

Statement (\ref{taaaa}) is in \cite{Buanetal07}.   Statement (\ref{taaa}) is in \cite{Balmer05}.  We can reinterpret this theorem using the topological language of the previous section.

\begin{definition}\label{btt}\hfill 
\begin{enumerate}
\item Let $\TT{\T^c}$ be the set of radical thick tensor ideals partially ordered by inclusion.
\item Let $\hspc \T$ denote the Hochster dual of $\spec \T$.  
\end{enumerate}
\end{definition}

By Theorem \ref{class}, the lattice $\TT{\T^c}$ is isomorphic to the lattice of Thomason subsets, and so we freely confuse the two.  But the latter is the collection of open sets of  $\hspc \T$.  So $\TT{\T^c}\cong\Fr{\hspc \T^c}$ and thus is a spatial frame.  Moreover, since $\hspc \T$ is  sober,  $\hspc\T\cong \spc \TT{\T^c}$.  See \cite{Kock07} or \cite{KockPitsch17} for a more thurough discussion.  


The situation is simplified if $\T$ satisfies a technical condition called \textit{rigid}. Since we will not use the condition itself, and only its consequences we refer the reader to \cite[Definition 1.3]{Stevenson16}.

\begin{lemma}[{\cite[Remark 1.18]{Stevenson16}}]\label{rigid}

If $\T$ is rigid, then every thick tensor ideal is radical. In this case, $\TT{\T^c}$ is the lattice of thick tensor ideals.  

\end{lemma}

\subsection{Localisation}

Let $(\T,\Sm,1,\tns)$ be a rigidly compactly generated tensor triangulated category. This means that $\T$ is generated by its compact objects $\T^c$ and that this category is essentially small, tensor closed, and is rigid.  Furthermore, we assume that the unit $1$ is compact.  See \cite[Section 1.1]{BalmerFavi11} for a discussion of these hypotheses.  The following are examples of such categories.
\begin{itemize}
\item The category $\T=D(R)$ for a commutative ring $R$.
\item The stable homotopy category.
\item The stable module category of $kG$ modules for $G$ a group whose order is divisible by the characteristic of $k$.  
\end{itemize}

\begin{definition}\ 
\begin{enumerate}
\item A subcategory $\cL\sbe\T$ is localising if it is thick and closed under arbitrary set-indexed coproducts.  A localising subcategory $\cL$ is a tensor ideal if for every $T\in \cL$ and $S\in \T$, $S\tns T$ is in $\cL$
\item For any Thomason subset $\wV\sbe \spc \T^c$, let $\T_\wV$ be the generated by the compact object $C\in \T^c$ such that
\[\Supp C\sbe \wV.\]
We call $\T_\wV$ the $\wV$-acyclic objects of $\T$.
\item An object $T\in \T$ is $\wV$-local if $\hm{\T}{\T_\wV}{T}=0$ or equivalently \[\hm{\T}{{\T_\wV}^c}{T}=0.\]
\end{enumerate}
\end{definition}

\begin{theorem}\label{idempotenttriangle}
Let $\wV\sbe \spc \T^c$ be a Thomason subset.  There exist triangulated coproducts preserving functors
\[\Ga_{\wV}\colon \T\to \T\quad\quad\quad L_\wV\colon\T\to\T\]
and natural transformations
\[\Ga_{\wV}\xto{\ga^{\wV}} \id\xto{\ld^\wV} L_\wV\xto{\eta^\wV} \Sm \Ga_\wV\]
which have the following properties.
\begin{enumerate}
\item For every $T\in \T$, the triangle
\[\Ga_{\wV} T\xto{\ga^{\wV}_T} T\xto{\ld^\wV_T} L_\wV T\xto{\eta^\wV_T} \Sm \Ga_\wV T\]
is exact.
\item The natural transformations
\[\Ga_{\wV}\ga^\wV\colon \Ga_{\wV}\Ga_{\wV}\to \Ga_{\wV}\quad\quad\quad L_\wV\ld^\wV \colon L_\wV \to L_\wV L_\wV \]
are isomorphisms.
\item The $\wV$-acyclic objects are equal to
\[\T_\wV=\im{\Ga_{\wV}}=\ker L_{\wV}.\]
Moreover, $\Ga_\wV$ is the identity on this category.
\item\label{acyc} The $\wV$-local objects are precisely the categories
\[\ker \Ga_{\wV}=\im{L_{\wV}}.\]
\item\label{tensorapp} The following functors are isomorphic
\[\Ga_{\wV}\cong -\tns \Ga_{\wV} 1\quad\quad\quad L_{\wV}\cong -\tns L_{\wV} 1.\]
\item\label{Intersections} If $\wV'\sbe\spc \T^c$ is another Thomason subset, then the following functors are isomorphic
\[\Ga_{\wV\cap \wV'}\cong \Ga_{\wV}\Ga_{\wV'}  \quad\quad\quad L_{\wV\cup \wV'}\cong L_{\wV}L_{\wV'}. \]
\item\label{main1h} For Thomason subsets $\wV,\wU,\wV',\wU'\sbe \spc\T^c$ such that
\[\wV\cap \wco{\wU}=\wV'\cap \wco{\wU'}\]
there is an isomorphism of functors
\[\Ga_\wV L_{\wU} \cong \Ga_{\wV'}L_{\wU'}. \]
\end{enumerate}
\end{theorem}

\begin{proof}

The category of $\wV$-acyclic objects is smashing; see \cite{Miller92} or \cite[Theorem 3.3.3]{Hoveyetal97}.  Statements (1)-(\ref{tensorapp}) are standard properties of smashing localisation; see \cite[4.9.1,4.10.1,4.11.1]{Krause10} and \cite[Definition 3.3.2]{Hoveyetal97} .  Statements (\ref{Intersections})  and (\ref{main1h}) are \cite[Theorem 5.18]{BalmerFavi11} and \cite[Corollary 7.5]{BalmerFavi11} respectively.
\end{proof}

We end this section with some notation.  

\begin{definition}\label{locypoky}
Recall from Example \ref{ZPThomason} that for any prime $\p\in \spc \T$, the set $\Ze{\p}$ is Thomason. Let $T$ be an object in $\T$.  
\begin{enumerate}
\item Set $T_\p=L_{\Ze{\p}}$.
\item Let $\Supp T$ be the primes $\p\in \spc \T^c$ such that $T_\p\ne 0$.
\end{enumerate}
\end{definition}


\begin{remark}\label{squib}

For every Thomason subset $\wV\sbe \spc\T^c$, the category ${\T_\wV}^c$ is the tensor ideal of $\T^c$ corresponding to $\wV$ in Theorem \ref{class} (\ref{taaa}).  Furthermore, the categories ${\T_{\Ze{\p}}}^c$ are the prime tensor ideals.  These ideals are also the meet irreducible elements of $\TT{\T^c}$.  
\end{remark}

\section{Support}\label{supportstuff}

In this section, $\T$ is a rigidly compactly generated tensor triangulated category.  

\subsection{Defining support}


\begin{definition}\hfill
\begin{enumerate}
\item For a $T\in \T$ and $\p\in \spc \T^c$, let $\p\in \supp T$ if for all Thomason subsets $\wV,\wU\sbe \spc \T^c$,  with $\p\in \wV\cap \wco{\wU}$ 
 \[\Ga_\wV L_{\wU} T\ne0.\]
 \item The localising space $\lspc \T$ will be the set $\spc \T^c$ with the topology generated by $\wV\cap\wco{\wU}$ with $\wV,\wU\sbe \spc \T^c$ Thomason.  In short 
 \[\lspc \T=\skula{\hspc\T^c}.\]
 We will refer to this topology as the localising topology.  
\end{enumerate}
\end{definition}

\begin{theorem}\label{main1}

Suppose $\T$ is a rigidly compactly generated tensor triangulated category.  The following are true.

\begin{enumerate}
\item\label{main1a} For an exact triangle
\[S\to T\to S'\to\]
we have 
\[\supp T\sbe \supp S\cap \supp S'\]
\item\label{main1b} If $T\cong S\op S'$ in $ \T$, then
\[\supp T=\supp S\cup \supp S'\]
\item\label{main1c} For any Thomason subset $\wV\sbe \spc\T^c$ and $T\in \T$, the following hold.
	\begin{enumerate}[(a)]
	\item $\supp \Ga_\wV T=\supp T\cap \wV$
	\item $\supp L_\wV T=\supp T\cap\wco{\wV}$
	\item There is an exact triangle
	\[T'\to T\to T''\to \]
	such that 
	\[\supp T'=\supp T\cap \wV \quad\quad\quad \supp T''=\supp T\cap \wco{\wV}.\]
	\end{enumerate}
\item\label{main1d} For any object $T\in \T$, we have $\supp T\sbe \Supp T$. Equality holds when $T$ is compact.  
\item\label{main1e} $\supp T$ is always closed in $\lspc \T$.
\item\label{main1f} For any localising subcategory $\cL \sbe \T$,  the set
\[\supp \cL=\bigcup_{T\in \cL} \supp T\]
is closed in $\lspc \T$.
\end{enumerate}

\end{theorem}

We devote the rest of this section  to proving this theorem.  

\begin{lemma}\label{subset}

 Consider Thomason subsets $\wV,\wU,\wV',\wU'\sbe \spc\T^c$ such that
\[\wV'\sbe \wV\quad\quad\quad  \wco{\wU'}\sbe \wco{\wU}.\]
If $\Ga_\wV L_{\wU} T=0$ then $\Ga_{\wV'}L_{\wU'}T=0.$  In particular $\p\in \supp T$ if and only if for all Thomason subsets $\wV\sbe \spc \T^c$ with $\p\in \wV$
\[\Ga_\wV T_\p\ne 0.\]

\end{lemma}

\begin{proof}

By Theorem \ref{idempotenttriangle} (\ref{Intersections}), the hypotheses imply that $\Ga_{\wV'}=\Ga_{\wV'}\Ga_\wV$ and ${L_{\wU'}=L_{\wU'}L_\wU}$.  Therefore, we have
\[\Ga_{\wV'}L_{\wU'}T=\Ga_{\wV'}L_{\wU'}\Ga_\wV L_{\wU} T=0.\]
\end{proof}

\begin{proof}[{Proof of Theorem \ref{main1} (\ref{main1a}) and (\ref{main1b})}]

Consider an exact triangle 
\[S\to T\to S'\to \]
and suppose that $\p\notin \supp S$ and $\p\notin \supp S'$.  Then there exists Thomason subsets $\wV,\wU,\wX,\wW$ such that 
\[\p\in \wV\cap \wco{\wU}\quad\quad\quad\p\in \wX\cap\wco{\wW}\]
\[\Ga_\wV L_{\wU} S=0 \quad\quad\quad \Ga_\wX L_{\wW} S'=0.\]
By the previous Lemma, we know that 
\[\Ga_{\wV\cap \wW} L_{\wco{(\wU\cup\wX)}} S=\Ga_{\wV\cap \wW} L_{\wco{(\wU\cup\wX)}} S'=0\]
and therefore $\Ga_{\wV\cap \wW} L_{\wco{(\wU\cup\wX)}} T=0$.
Furthermore, 
\[\p\in \wV \cap \wco{\wU} \cap \wW\cap\wco{\wX}=\wV\cap \wW\cap \wco{(\wU\cup\wX)}.\]
Therefore, $\p\notin \supp T$, proving Theorem \ref{main1} (\ref{main1a}).

Now we assume that $T\cong S\op S'$.  To prove Theorem \ref{main1} (\ref{main1b}),  observe that $\Ga_\wV L_{\wU} T=0$ if and only if $\Ga_\wV L_{\wU} S=0$ and $\Ga_\wV L_{\wU} S'=0$  for any Thomason subsets $\wV,\wU\sbe \spc\T^c$.
\end{proof}

\begin{proof}[{Proof of Theorem \ref{main1} (\ref{main1c})}]

Let $\wV \sbe \spc\T^c$ be a Thomason subset and let $T\in \T$.  The vanishing objects
\[\Ga_{\spc\T^c}L_{\wV}\left(\Ga_{\wV} T\right)=0\quad\quad\quad \Ga_{\wV}L_\emptyset \left(L_{\wV} T\right)= \Ga_{\wV}L_{\wV} T=0\]
 implies that $\supp \Ga_{\wV} T\sbe \wV$ and $\supp L_{\wV} T\sbe \wco{\wV}$.

The idempotent triangle 
\[\Ga_{\wV} T\to T\to \ L_{\wV} T\to.\]
and Theorem \ref{main1} (\ref{main1a}) imply the containment
\[\supp \Ga_{\wV} T\sbe \supp T\cup \supp L_{\wV} T\sbe  \supp T\cup\wco{\wV}.\]
Intersecting both sides with $\wV$ yields
\[\supp \Ga_{\wV} T=\supp T\cap \wV\]
 proving Theorem \ref{main1} (\ref{main1c}a). A similar calculation proves  Theorem \ref{main1} (\ref{main1c}b).
Setting $T'=\Ga_{\wV} T$ and $T''=L_{\wV} T$ proves Theorem \ref{main1} (\ref{main1c}c).
\end{proof}

\begin{proof}[{Proof of Theorem \ref{main1} (\ref{main1d})}]

Let $T\in \T$.  If $\p\notin \Supp T$, then 
\[T_\p=\Ga_{\spc\T^c} L_{\Ze{\p}} T=0\] 
and so $\p\notin \supp \T$, which implies the containment $\supp T\sbe \Supp T$.

  Now suppose $T$ is compact.    To show the reverse containment, suppose that $\p\notin \supp \T$.  Then there exists a Thomason subset $\wV\sbe \spc\T^c$  such that $\p\in \wV$ and  $\Ga_{\wV} T_\p= 0$.  Let $\cI$ be the thick tensor ideal of $\T^c$ generated by $T$.  Since $\T^c$ is rigid, $\cI$ is automatically radical by Lemma \ref{rigid}.  Since every  $S\in \T^c_{\wV}\cap\cI$ satisfies 
  \[{S}_\p\cong\Ga_{\wV} S_\p=0\]
we have 
  \[{\T_{\wV}}^c\cap \cI\sbe{\T_{\Ze{\p}}}^c.\]
Since $\p\in \wV$, we know that $\wV\nsubseteq \Ze{\p}$ and so ${\T_{\wV}}^c\nsubseteq{\T_{\Ze{\p}}}^c$  by Theorem \ref{class} and Remark \ref{squib}. Since ${\T_{\Ze{\p}}}^c$ is a prime tensor ideal, we conclude that $\cI\sbe {\T_{\Ze{\p}}}^c$.  Thus $T_\p=L_{\Ze{\p}} T=0$.  
\end{proof}

\begin{proof}[{Proof of Theorem \ref{main1} (\ref{main1e})}]

Suppose $\p\notin\supp T$ for some $T\in \T$.  Then there exists Thomason subsets $\wV,\wU\sbe\spc\T^c$ such that $ \Ga_{\wV}L_{\wU} T=0$ and  $\p\in \wV\cap\wco{\wU}$. Now consider any $\q\in\wV\cap \wco{\wU}$. Then $\q$ is also not in $\supp T$.  Therefore, there is an open neighbourhood of $\p$ in $\lspc \T$ which is not in $\supp T$.  It follows that $\supp T$ is closed.  
\end{proof}

\begin{proof}[{Proof of Theorem \ref{main1} (\ref{main1f})}]

Let $\cL\sbe \T$ be a localising subcategory.  For every $\p\in \supp \cL$, choose some element $T^\p$ such that $\p\in \supp T^\p$.   Set
\[T=\coprod_{\p\in \supp \cL} T^\p\]
We claim that 
\[\supp \cL=\supp T.\]
Given the claim, the result follows from Theorem \ref{main1} (\ref{main1e}). First, $\supp T\sbe \supp \cL$ since $T\in \cL$.  For any $\q\in \supp\cL$ and any Thomason closed subsets $\wV,\wU\sbe\spc\T^c$ with $\p\in \wV\cap\wco{\wU}$, we have
\[\Ga_{\wV}L_{\wU} T=\Ga_{\wV}L_{\wU} \coprod_{\p\in \supp \cL} T^\p\cong\Ga_{\wV}L_{\wU} T^\q\sqcup \coprod_{\substack{\p\in \supp \cL\\ \p\ne \q}} \Ga_{\wV}L_{\wU} T^\p\ne0.\]
Thus $\p\in \supp T$. 
\end{proof}

\subsection{Visible points}\label{vis}

In this section, we relate our support to that of Balmer, Favi, and Stevenson in \cite{BalmerFavi11,Stevenson13b}.  Following \cite{Stevenson13b},  a point $\p$ in a spectral space  $\wX$ is {\it visible} if there exists Thomason subsets $\wV,\wU\sbe\wX$ such that
\[\{\p\}=\wV\cap\wco{\wU}.\]
This definition is more general than \cite[Definition 7.9]{BalmerFavi11}.

\begin{lemma}\label{Noeth}

The following are equivalent for a spectral space $\wX$.  
\begin{enumerate}
\item All points of $\wX$ are visible.
\item the localising topology on $\wX$ is discrete.
\item The Hochster dual $\wX^\da$ is $\sT_{\frac{1}{2}}$
\end{enumerate}
Moreover, these conditions hold when $\wX$ is Noetherian.
\end{lemma}
See Definition \ref{skuladef} to recall the oft forgotten seperation axiom $\sT_{\frac{1}{2}}$.
\begin{proof}

The equivalence  is straightforward. See \cite[Proposition 7.13]{BalmerFavi11} for the last statement.  
\end{proof}

Let $\T$ be a rigidly compactly generated tensor triangulated category.   When  a prime  $\p\in \spc\T^c$  is visible and write $\{\p\}=\wV\cap  \wco{\wU}$ for Thomason subsets $\wV,\wU$.   Define the {\it Rickard idempotent}
\[\Ga_\p 1=\Ga_\wV1\tns L_\wU 1.\]
By \cite[Corollary 7.5]{BalmerFavi11}, this definition is independent of the choice of $\wV$ and $\wU$.  Rickard idempotents have appeared in  \cite{Rickard97,BensonIyengarKrause08} and other works.  When every prime is visible, then $\p\in \supp T$ if and only if 
\[\Ga_\p 1\tns T\ne 0.\]
Thus  our definition of support recovers \cite[Definition 7.16]{BalmerFavi11} and \cite[Definition 5.6]{Stevenson13b}. Furthermore, Theorem \ref{main1} is a generalisation of \cite[Theorem 7.17, Proposition 7.18]{BalmerFavi11}.


For  $R$ a commutative Noetherian ring, $\T=\DD{R}$ is a compactly generated tensor triangulated category.  In this case, $\spc \DD{R}^c=\spec R$ by the Hopkins Neeman theorem \cite{Hopkins87,Neeman92}.  Every point is visible in $\spec R$  and so our support coincides with that of Balmer and Favi.  Moreover by the work of Foxby and Iyengar in \cite{Foxby79} and \cite[2.1 and 4.1]{FoxbyIyengar03}, $\p\in \supp M$ for $M\in \DD{R}$ if and only if $M\lns k(\p)\ne 0$ where $k(\p)$ is the residue field at $\p$.

\subsection{Detecting vanishing}\label{sup}

\begin{definition}

We call $\T$ supportive if $T=0$ if and only if $\supp T=\emptyset$.

\end{definition}

\begin{example}

By \cite[Lemma 2.6]{Foxby79} and \cite{Neeman92}, $\DD{R}$ is supportive for all Noetherian commutative rings $R$.  

\end{example}

\begin{theorem}\label{main2}

Let $\T$ be a rigidly compactly generated tensor triangulated category.  If $\T$ is supportive, the following are true.

\begin{enumerate}
\item\label{main2c}For any Thomason subset $\wV\sbe \spc\T^c$ and $T\in \T$,  the following statements hold.
	\begin{enumerate}[(a)]
	\item $T$ is $\wV$-acyclic if and only if $\supp T\sbe \wV$
	\item $T$ is $\wV$-local if and only if $\supp T\sbe\wco{\wV}$
	\item $\T_\wV=\{T\in \T\mid \Supp T\sbe \wV\}$
	\end{enumerate}
\item\label{main2b} If there exists a Thomason subset $\wV\sbe \spc\T^c$ such that $\supp S\sbe \wV$ and $\wV\cap \supp T=\emptyset$, then $\hm{\T}{S}{T}=0$.  
\item\label{main1c'} For any Thomason subsets $\wV,\wU\sbe\spc\T^c$ and $T\in \T$, then 
\[\wV\cap\wco{\wU}\cap \supp T=\emptyset\]
if and only if $\Ga_\wV L_{\wU} T=0$.
\item\label{main2g'} For any $\cX\sbe \T$ let $\loc^\tns \cX$ be the localising tensor ideal closure of $\cX$.  Then
\[\supp \loc^\tns \cX=\overline{\supp \cX}\]
where the closure is taken in the localising topology.  In particular, for any set $\{T_\al\}\sbe \T$, we have
\[\supp \coprod T_\al =\overline{\bigcup \supp T_\al}\]

\item\label{main1g} For any close set $\wX\sbe \lspc\T$, the category
\[\{T\in \T\mid \supp T\sbe \wX\}\]
is localising.

\end{enumerate}

\end{theorem}

\begin{proof}[{Proof of Theorem \ref{main2} (\ref{main2c})}]

We prove  Theorem \ref{main2} (\ref{main2c}a).  By by Theorem \ref{main1} (\ref{main1c}a), if $T$ is $\wV$-acyclic then $\supp T\sbe \wV$.  We now prove the converse.  If $\supp T\sbe \wV$, then by Theorem \ref{main1} (\ref{main1c}b),
\[\supp L_{\wV}T=\supp T\cap \wco{\wV}=\emptyset.\]
If $\T$ is supportive, then $L_{\wV} T=0$ and so $T$ is $\wV$-acyclic.  A similar argument shows Theorem \ref{main2} (\ref{main2c}b).

We now prove  Theorem \ref{main2} (\ref{main1c}c).  Set 
\[\cL=\{T\in \T\mid \Supp T\sbe \wV\}.\]
 By Theorem \ref{main2} (\ref{main2c}a), $\T_\wV=\{T\in\T\mid \supp T\sbe \wV\}$.  By Theorem \ref{main1} (\ref{main1d}), $\supp T\sbe \Supp T$ and so $\cL\sbe \T_\wV$.   By definition, ${\T_\wV}^c\sbe \cL$.  Since $\cL$ is localising, it follows that $\T_\wV\sbe \cL$.  
\end{proof}

\begin{proof}[{Proof of Theorem \ref{main2} (\ref{main2b})}]

If there exists a Thomason subset $\wV\sbe \spc\T^c$ such that $\supp S\sbe \wV$ and $\wV\cap \supp T=\emptyset$, then $S$ is $\wV$-acyclic and $T$ is $\wV$-local Theorem by \ref{main2} (\ref{main2c}).  It follows that $\hm{\T}{S}{T}=0$.  
\end{proof}

\begin{proof}[{Proof of Theorem \ref{main2} (\ref{main1c'})}]

This follows by applying Theorem \ref{main1} (\ref{main1c}a) and (\ref{main1c}b) and the supportive condition.
\end{proof}

\begin{proof}[{Proof of Theorem \ref{main2} (\ref{main2g'})}]

Let $\cX\sbe\T$ be a collection of objects.  It is clear that $\supp \cX\sbe \supp \loc \cX$.  By Theorem \ref{main1} (\ref{main1f}), we know that $\supp \loc^\tns \cX$ is closed, and so
\[\overline{\supp \cX}\sbe \supp \loc^\tns \cX.\]
Now take a prime $\p\notin \overline{\supp \cX}$.  There exist Thomason subsets $\wV,\wU\sbe \spc\T^c$ such that  $\p\in \wV\cap\wco{\wU}$ and $ \wV\cap\wco{\wU}$ is disjoint from $\supp\cX$. From Theorem \ref{main2} (\ref{main1c'}),  
\[\Ga_\wV L_{\wU} \cX=0.\] 
 Since the kernel of $\Ga_\wV L_{\wU}$ is a tensor ideal,  $\Ga_\wV L_{\wU} \loc^\tns \cX=0$.  Therefore $\p\notin \supp \loc^\tns\cX$.  
\end{proof}

\begin{proof}[{Proof of Theorem \ref{main2} (\ref{main1g})}]

Let $\wX\sbe\spc\T^c$ be closed in the localising topology.  Set
\[\cL=\{T\in \T\mid \supp \T\sbe \wX\}.\]
Theorem \ref{main1} (\ref{main1a}) shows that $\cL$ is thick.  For a set $\{T_\al\}\sbe \cL$ Theorem \ref{main2} (\ref{main2g'}) implies that
\[\supp \coprod T_\al=\overline{\bigcup \supp T_\al}\sbe \wX\]
since $\wX$ is closed.  It is easy to check that $\cL$ is a tensor ideal.  
\end{proof}

 \subsection{Characterization of support}
 
 It is not clear if $\T$ is always supportive.  However, the following result tells us that if $\supp$ does not detect vanishing, then there is no other reasonable support function taking values in $\spc\T^c$ that will.

\begin{theorem}\label{main3}
Consider the following properties of a function 
\[\ws\colon:\T\to \{\mbox{Subsets of } \spc \T^c\}.\]
Let $T\in \T$.
\begin{enumerate}
\item\label{Vanishing} Vanishing: $\ws(T)=\empty$ if and only if $T=0$. 
\item\label{Local} Local: If $\wV\sbe\spc \T^c$ is Thomason and $T$ is $\wV$-local, then $\ws(T)\cap \wV=\emptyset$.
\item\label{Big Support} Big Support: $\ws(T)\sbe \Supp T$ for any compact object $T\in \T^c$.
\item\label{Orthogonality} Orthogonality: For any $S\in \T$, if there is a Thomason subset $\wV\sbe\spc \T^c$  such that $\ws(T)\sbe \wV$ and $\ws(S)\cap \wV=\emptyset$, then 
\[\hm{\T}{T}{S}=0.\]
\item\label{Exactness} Exactness: For any exact triangle $S\to T\to S'\to$ in $\T$,
\[\ws(T)\sbe \ws(S)\cup\ws(S').\]
\item\label{Separation} Separation: For any Thomason subset $\wV\sbe\spc \T^c$, there is an exact triangle
\[T'\to T\to T''\to\]
with $\ws(T')\sbe \wV$ and $\ws(T'')\cap\wV=\emptyset$.
\end{enumerate}
The function $\ws$ satisfies all of these properties, if and only if $\ws=\supp$ and $\T$ is supportive.  
\end{theorem}

The proof is similar to that of \cite[Theorem 5.15]{BensonIyengarKrause08}.

\begin{proof}
If $\T$ is supportive, then $\supp$ satisfies these properties by Theorem \ref{main1} and Theorem \ref{main2}.  Conversely,  suppose $\ws$ satisfies all of these properties. We need to show that $\ws=\supp$.  

Let $T\in \T$, and consider the exact triangle
\[T'\to T\to T''\to\]
guaranteed by (\ref{Separation}). By (\ref{Exactness}) and (\ref{Separation}), $\ws(T')=\ws(T)\cap \wV$ and $\ws(T'')=\ws(T)\cap\wco{\wV}$.  

We claim that 
\[T'\cong \Ga_{\wV} T \quad\quad \mbox{and} \quad\quad T''\cong L_{\wV} T\]
for any Thomason subset $\wV\sbe\spc \T^c$.  Given the claim,  $\p\notin \ws(T)$ if and only if there are Thomason subsets $\wV,\wU\sbe \spc\T^c$ with $\p\in \wV\cap\wco{\wU}$ such that 
\[\ws\left(\Ga_{\wV}L_{\wU} T\right)=\emptyset.\]
Therefore, the claim and (\ref{Vanishing}) imply that $\ws(T)=\supp T$.  

We now prove the claim.  Let $X\in \T$ be a $\wV$-local object and let $Y\in {\T_\wV}^c$. Then by (\ref{Local}) and (\ref{Big Support})
\[\ws(X)\cap \wV=\emptyset\quad\quad\mbox{and}\quad\quad \ws(Y)\sbe \Supp Y\sbe \wV.\]
Then (\ref{Orthogonality}) and (\ref{Separation})  imply 
\[\hm{\T}{T'}{X}=0\quad\quad\mbox{and}\quad\quad \hm{\T}{Y}{T''}=0.\]
We conclude that $T'$ is $\wV$-acyclic and $T''$ is $\wV$-local by \cite[Proposition 4.10.1]{Krause10} and the definition of $\wV$-local.  

By \cite[Proposition 4.11.2]{Krause10}, we have the following commutative diagram.
\[\xymatrix{
T' \ar[r] \ar[d]^{f} & T \ar[r] \ar@{=}[d]  & T'' \ar[r] \ar[d]^{g} &\Sm T' \ar[d]^{\Sm f}\\
\Ga_{\wV} T \ar[r] & T \ar[r] & L_\wV  T \ar[r] & \Sm \Ga_{\wV} T\\
}\]
By the octahedral axiom, the cones
\[\cone g\cong \Sm \cone f\]
are isomorphic and thus both are $\wV$-acyclic and $\wV$-local.  Hence the cones are zero, and so $f$ and $g$ are isomorphisms. 
\end{proof}

\section{Commutative rings}\label{commstuff}

\subsection{Support}

 In this section, we specialise our theory of supports to the derived category  $\DD{R}$  for a commutative ring $R$.  
 For $\underline{x}=x_1,\dots,x_n\in R$, we write
 \[\KK{\infty}{\underline{x}}=(R\to R_{x_1})\tns \cdots \tns(R\to R_{x_n})=\]
\[0\to R\to \bigoplus_{1\le i\le n} R_{x_i}\to \bigoplus_{1\le i<j\le n} R_{x_ix_j} \to \cdots \to \bigoplus R_{x_1\cdots x_n}\to 0.\]

%
%
%
%
%

\begin{lemma}\label{zappy}\ 
\begin{enumerate}
\item\label{zappy0} The category $\DD{R}$ is a  rigidly compactly generated tensor triangulated category whose tensor product is $\lns$.
\item\label{zappy1} The compact objects of $\DD{R}$ are the { \it perfect complexes}, i.e.\ complexes which are quasi-isomorphic to bounded complexes of finitely generated free modules.  We will denote the perfect complexes by $\perf(R)$. 
\item\label{zappy2}  We have $\spc \perf(R)=\spec R$.  
\item\label{zappy3} The closed sets of $\spec R$ with quasi-compact complement are those of the form $V(\underline{x})$ with $\underline{x}=x_1,\dots,x_n\in R$.  Thus, every Thomason set is a union of such sets.  Moreover, 
\[\Ga_{\wV(\underline{x})} R=\KK{\infty}{\underline{x}}.\]
When $R$ is Noetherian this is $\wR\Ga_{\underline{x}} R$, the right derived torsion functor applied to $R$.  
\item\label{zappy4}  For any $\p\in \spec R$ and $M\in \DD{R}$, 
\[L_{\Ze{\p}} M=M_\p.\]
Hence the notation in Definition \ref{locypoky} is unambiguous.  
\end{enumerate}
\end{lemma}

\begin{proof}

 For (\ref{zappy1}), See \cite[Example 1.10,1.13]{Neeman96}.  Statement (\ref{zappy2}) is the Hopkins, Neeman, Thomason theorem \cite{Neeman92}, \cite{Hopkins87}, and \cite{Thomason97}.  For (\ref{zappy3}), see \cite[Lemma 5.8]{Greenlees01}, although the argument stems from the Noetherian case treated in \cite{Hartshorne67}. 
 
 We show (\ref{zappy4}). Any $M\in {\DD{R}}_{\bZ(\p)}$ satisfies  $M_\p=0$, and thus $L_{\Ze{\p}} R_\p\cong R_\p$.  Furthermore, the complex 
\[0\to R\to R_\p\to 0\]
is a direct limit of complexes 
\[\Ga_{\wV(x)} R=0\to R\to R_x\to 0\]
with $x\notin \p$.  Since $L_{\Ze{\p}}\Ga_{\wV(x)} =0$, the first complex is in the kernel of $L_{\Ze{\p}}$.  It follows that $L_{\Ze{\p}} R\cong L_{\Ze{\p}} R_\p$.  
\end{proof}

We can now restate our definition of support.  In fact, for the reader whose interest only lies in commutative algebra, the following can be taken as a definition.  
For a complex $M\in\DD{R}$ and a sequence $\underline{x}\in R$,  set
\[\Kinf{\underline{x}}{M}=\KK{\infty}{\underline{x}}\tns M\quad\quad\mbox{and}\quad\quad\HKinf{i}{\underline{x}}{M}=\HH{i}{}{\Kinf{\underline{x}}{M}}.\]

%
%

\begin{lemma}\label{ringsup}

For a complex $M\in \DD{R}$, a prime $\p\in \supp M$ if and only if for every finite sequence $\underline{x}=x_1,\dots,x_n\in \p$
\[\HKinf{i}{\underline{x}}{ M_\p}\ne 0\]
for some $i$.  
\end{lemma}

\begin{proof}
%
This result follows directly from  Lemma \ref{subset} and Lemma \ref{zappy} (\ref{zappy3}). 
\end{proof}

\subsection{Detecting vanishing}

We now discuss the supportive condition in $\DD{R}$.  
\begin{definition}

For a module $M\in \Md(R)$,  a prime $\p\in\spec R$  is weakly associated to a module $M$ if there exists an element $m\in M$ such that $\p$ is minimal amongst primes containing $\ann(m)$.  Let $\wass_R M$ denote the the set of weakly associated primes of $M$.  Let $\mid\wass_R M$ be the minimal weakly associated primes.  

\end{definition}

The following result is an exercise in  \cite[IV,1, Exercise 17]{Bourbaki89}.  See \cite{Merker69} for proofs.  

\begin{lemma}\label{weakass}
Let $M$ be an $R$-module and $\p\in \spec R$ a prime.  
\begin{enumerate}
\item\label{weakassa} $\p\in \wass_R M$ if and only if $\p R_\p\in \wass_{R_\p} M_\p$
\item\label{weakassb} $M=0$ if and only if $\wass_R M=\emptyset$
\item\label{weakassc} When $R$ is Noetherian, then $\wass M=\ass M$
\item\label{weakassd} If $W\sbe R$ then 
\[\wass \Ga_W M=\{\p\in \wass M\mid \p\cap W\ne \emptyset\}\]
where $\Ga_W M$ is the $W$-torsion submodule of $M$. 
\item If $W\sbe R$ is multiplicatively closed, then 
\[\wass\label{weakasse} M_W=\{\p\in \wass M\mid \p\cap W= \emptyset\}.\]
\end{enumerate}
\end{lemma}

\begin{theorem}\label{main6}\hfill
\begin{enumerate}
\item\label{main6a} Let $M\in \DD{R}$ be a complex of $R$-modules such that $\HH{i< n}{}{M}=0$, e.g.\ $M$ is a module.  Then 
\[ \wass_R \HH{n}{}{M}\sbe \supp M.\]
In particular, $M=0$ if and only if $\supp M=\emptyset$.  
\item\label{main6b} For any complex $M\in \DD{R}$
\[ \min\wass_R \bigoplus_{i\in \ZZ} \HH{i}{}{M}\sbe \supp M.\]
\item\label{main6c} If  the prime ideals of $R$ satisfy DCC, then $\DD{R}$ is supportive, i.e.\ \break$\supp M=\emptyset$ if and only if $M=0$.  
\end{enumerate}
\end{theorem}

Note, we will write $\supp_R M$ for $\supp M$ if there is any confusion over which ring we are computing the support.  

\begin{proof}[Proof of Theorem \ref{main6} (\ref{main6a})]

We may assume $n=0$.  Suppose $\p\in \wass_R \HH{0}{}{M}$.    By Lemma \ref{ringsup}, it suffices to show that $\p R_\p\in \supp_{R_\p} M_\p$. Since $\HH{0}{}{M}_\p\cong \HH{0}{}{M_\p}$, Lemma \ref{weakass} (1) implies that $\p R_\p\in \wass \HH{0}{}{M_\p}$.  Therefore, we can assume that $R$ is local with maximal ideal $\mm$ and that $\p=\mm$.  

We claim that for every $x\in \mm$, the exact triangle $\Ga_{\wV(x)} M\to M\to M_x\to$ yields
\[\cdots\to 0\to \HKinf{0}{x}{M}\to \HH{0}{}{M}\to \HH{0}{}{M}_x\to \cdots.\]
Therefore $\HKinf{i<0}{x}{M}=0$ and 
\[\HKinf{0}{x}{M}=\HHbig{0}{}{\Kinf{x}{M}}=\Ga_x \HH{0}{}{M}.\]  
By induction, for any $\underline{x}=x_1,\dots,x_n\in \mm$
\[\HKinf{0}{\underline{x}}{M}=\HKinfbig{0}{x_n}{\Kinf{x_1,\dots,x_{n-1}}{M}}=\Ga_{\underline{x}} \HH{0}{}{M}.\]
By Lemma \ref{weakass} (4), $\mm\in \wass \Hck{0}{\underline{x}}{M}$, and so $\Hck{0}{\underline{x}}{M}$ does not vanish by Lemma \ref{weakass} (2).
\end{proof}

\begin{proof}[Proof of Theorem \ref{main6} (\ref{main6b}) and (\ref{main6c})]
We write $\HH{}{}{M}$ for $\bigoplus \HH{i}{}{M}$.  If the primes ideals of $R$ satisfy DCC, then $\HH{}{}{M}$ has a minimal weakly associated prime for all nonzero $M\in \DD{R}$, by Lemma \ref{weakass} (2).  Thus (\ref{main6b}) implies (\ref{main6c})

We now show (\ref{main6b}).  Let $\p\in \min \wass \HH{}{}{M}$.  As in the proof of (\ref{main6a}),  $\p R_\p$ is a weakly associated prime of $\HH{}{}{M_\p}$.    By Lemma \ref{weakass} (\ref{weakasse}), $\p R_\p$ is a minimal weakly associated prime.  Thus, we assume that $R$ is a local ring with maximal ideal $\mm$, and that $\p=\mm$.

Since $\mm$ is the only weakly associated prime of $\HH{}{}{M}$. In this case, for any element of $\mu\in \HH{}{}{M}$, $\mm$ is the only prime minimal over $\ann \mu$, and so $\sqrt{\ann \mu}=\mm$.  Therefore every element of $\HH{}{}{M}$ is $\mm$-torsion.  It follows that that $\HH{}{}{M}_x=0$ for any $x\in \mm$, and therefore 
\[\Kinf{x}{M}=(R\to R_x)\tns M\cong M=M.\]
It follows that that $\Kinf{\underline{x}}{M_\mm}\cong M\ne 0$ for all finite sequences $\underline{x}\in M$, and so $\mm\in \supp M$.  
\end{proof}

\begin{example}
Set 
\[R=\frac{k[x_2,x_3,x_4,\cdots]}{({x_2}^2,{x_3}^3,{x_4}^4,\cdots)}.\]
In \cite{Neeman00}, Neeman shows that the collection of localising subcategories of $\DD{R}$ is atrocious.   However, the support theory is simple:  $\spec R$ has one prime ideal $\mm=(x_2,x_3,x_4,\dots)$ and so $\DD{R}$ is supportive by the previous theorem.  For $M\in \DD{R}$, then  $\supp M=\{\mm\}$  unless $M$ is acyclic.  
\end{example}

\begin{example}\label{Jan}
We now give an example of a ring such that $\DD{R}$ is supportive, but Foxby's support does not detect vanishing.  In \cite[Example 5.34]{Stovicek17}, \v{S}\v{t}ov\'{i}\v{c}ek describes a commutative ring $R$ with the following properties.  First, the ring $R$ is a valuation domain with prime ideals $0$ and $\mm$.  Furthermore, $\tor_{>0}(k,k)=0$ where  $k=R/\mm$, then.  Let $Q$ be the quotient field of $R$.  Note that telescope conjecture fails for $\DD{R}$ by \cite{Keller94}.  Also $\DD{R}$ is supportive since $\spec R$ satisfies DCC on prime ideals.  

Let $M$ be the cokernel of the composition
\[\mm\hookrightarrow R\hookrightarrow Q.\]
We claim that $M\lns k(\p)=0$ for all $\p\in \spec R$, i.e. Foxby's support of $M$ is empty.  Indeed, $M$ is quasi-isomorphic to the complex
\[(R\to k)\lns (R\to Q).\]
Since $R$ has only two prime ideals $k(\p)$ can either be $k$ or $Q$.  However $(R\to k)\tns k$ and $(R\to Q)\lns Q$ are both acyclic, proving the claim.
\end{example}

\subsection{Properties of support}
In this section we show that support over non-Noetherian rings behaves similarly to support over Noetherian rings.  Note that a careful examination of the proofs of Section \ref{sup} show that even though we do not have the supportive condition for all complexes, the results still hold for $\DD{R}$.  

\begin{prop}
Let $R$ be a ring, $\HH{\ll 0}{}{M}$ and $W\sbe R$ be a multiplicatively closed subset.  Then 
\[\supp M_W=\{\p\in \supp M\mid \p\cap W=\emptyset.\}\]
\end{prop}

\begin{proof}
 For any $\p$ that does not intersect $W$, we have $M_\p={(M_W)}_\p$ and so\break $\p\in \supp M_W$ if and only if $\p\in \supp M$.  If there is a $w\in \p\cap W$ then the equality
\[\Ga_{\wV(w)} {(M_W)}_\p=0\]
implies $\p\notin\supp M$.  
\end{proof}

\begin{prop}\label{trans}
Let $f\colon R\to S$ be homomorphism of commutative rings, and $M\in \DD{S}$.  
\begin{enumerate}
\item For any Thomason subset $\wV\sbe \spec R$, the morphism ${\aspec f }^{-1}(V)$ is Thomason.  In particular, the $\aspec f $ is continuous in both the Hochster dual topology and the localising topology.  Furthermore,  the following are isomorphisms in $D(R)$
\[ \Ga_{\wV} M\cong \Ga_{{\aspec f }^{-1}(\wV)} M\quad\quad\quad L_{\wV} M\cong L_{{\aspec f }^{-1}(\wV)} M.\]
\item All Thomason subsets $\wV,\wU\sbe \spec R$ satisfy
\[\supp_S M\cap {\aspec f }^{-1}(\wV\cap\wco{\wU})=\supp_S \Ga_{{\aspec f }^{-1}(\wV)}L_{{\aspec f }^{-1}(\wU)} M.\]
\item If $\HH{\ll 0}{}{M}$, then \[\aspec f (\supp_ S M)=\supp_R M.\]
\end{enumerate}
\end{prop}

\begin{proof}
For $\underline{x}=x_1,\dots,x_n\in R$, we have ${\aspec f }^{-1}(\wV(\underline{x}))=\wV(d(\underline{x}))$.  This proves the first statement of (1).  Consider the idempotent triangle
\[\Ga_{\wV(\underline{x})} R\to R\to L_{\wV(\underline{x})}  R\to\]
and apply $S\lns-$.  By Lemma \ref{zappy} (\ref{zappy3}) and Theorem \ref{idempotenttriangle} (\ref{tensorapp}), we have 
\[\Ga_{\wV(\underline{x})} S\cong \Ga_{\wV(f(\underline{x}))} S\cong \Ga_{{\aspec f }^{-1}(\wV(\underline{x}))} S\quad\quad\quad L_{\wV(\underline{x})} S\cong L_{\wV(f(\underline{x}))} S\cong L_{{\aspec f }^{-1}(\wV(\underline{x}))} S.\]
Statement (1) follows for $M=S$  now follows by taking the direct limit of the above functors over all $\wV(\underline{x})\sbe \wV$.  Note that this result applies to all 

Statement (2) follows from (1) and Theorem \ref{main1} (\ref{main1c}).  We now prove (3).   Let $\p\in \spec R$ and $\wV,\wU\in \spec R$ such that $\p\in \wV\cap\wco{\wU}$.  Statement  (2) and the supportive imply that the  following are equivalent
\[\Ga_{\wV} L_{\wU} M=0 \iff \supp_S  \Ga_{{\aspec f }^{-1}(\wV)}L_{{\aspec f }^{-1}(\wU)} M=\emptyset\iff \supp_S M\cap {\aspec f }^{-1}(\wV\cap\wco{\wU})=\emptyset.\]
\end{proof}

\begin{remark}

Proposition \ref{trans} (3) holds for all $M\in \DD{S}$ if $\DD{S}$ is supportive.  In particular, if $\DD{S}$ is supportive and $f\colon R\to S$ is faithfully flat, then $\DD{R}$ is supportive.

\end{remark}

\section{Support and the Bousfield lattice}\label{framestuff}



Again, let $\T$ be a rigidly compactly generated tensor triangulated category.  

\begin{definition}\hfill
\begin{enumerate}
\item For any element $T\in \T$, set
\[\Ac{T}=\{S\in \T\mid T\tns S=0\}.\]
Any class of the form $\Ac{T}$ is called Bousfield.  
\item A Bousfield class $\Ac{T}$ is idempotent if $\Ac{T}=\Ac{T\tns T}$.  
\item Let $\Ab{\T}$ denote the collection of idempotent Bousfield classes.  Order $\Ab{\T}$ via reverse inclusion.  The elements $\Ac{1}$  and $\Ac{0}$ are the maximum and minimum respectively.  
\end{enumerate}
\end{definition}

For every pair of Thomason subsets $\wV,\wU\sbe \spc\T^c$, the kernel of the functor $\Ga_\wV L_\wU$ is precisely the Bousfield class $\Ac{\Ga_\wV L_\wU 1}$ by Theorem \ref{idempotenttriangle} (\ref{tensorapp}).  Moreover, since $\Ga_\wV L_\wU 1$ is idempotent, this Bousfield class is also idempotent.

The first statement of the following theorem was originally proven by Ohkawa in \cite{Ohkawa89} for the stable homotopy category.  The second statement was also proven in this context by Bousfield in \cite{Bousfield79} and by Hovey and Palmieri in \cite[Proposition 3.4]{HoveyPalmieri98}.  The following result was proven in our generality by Krause and Iyengar in \cite[Theorem 3.1, Proposition 6.2]{IyengarKrause13}.

\begin{theorem}

The class $\Ab{\T}$ is a set.  Moreover, it is a frame where arbitrary joins and finite meets are given by
\[\bigjn_i \Ac{T_i}=\bigcap_i\Ac{T_i}=\Ac{\coprod_i T_i}\quad\quad\quad \Ac{T_1}\mt\cdots\mt\Ac{T_n}=\Ac{T_1\tns\cdots\tns T_n}\]
\end{theorem}

Pointless topology  now creeps into our theory.   Suppose $\supp T=\emptyset$.  Then for every $\p\in \spc\T^c$ there is a pair of Thomason subsets $\wV_\p,\wU_\p$ with $\p\in \wV_\p\cap \wco{\wU_\p}$ with $\Ga_{\wV_\p}L_{\wU_\p} T=0$, or in other words with $T\in \Ac{\Ga_{\wV_\p}VL_{\wU_\p}  1}$.  Hence  
\[T\in \bigjn_{\p\in \spc\T^c} \Ac{\Ga_{\wV_\p}VL_{\wU_\p}  1}.\]
Now $\T$ is supportive precisely when the above Bousfield class always vanishes, i.e.
\[\bigjn_{\p\in \spc\T^c} \Ac{\Ga_{\wV_\p} L_{\wU_\p}  1}=0=\Ac{1}=\Ac{\Ga_{\spc\T^c} L_\emptyset 1}.\]
Compare this equality the union of basic open sets in $\lspc \T$
\[\bigcup_{\p\in \spc\T^c} \wV_\p\cap \wco{\wU_\p}=\spc\T^c\cap \wco{\emptyset}.\]
 If two unions of basic open neighbourhoods in the Skula topology are set theoretically equal,  do they define the same idempotent Bousfield class? If this is so, then $\T$ is supportive.  As we will see in Theorem \ref{main4}, the converse is also true.

Recall from Definition \ref{btt} and the following discussion that $\TT{\T^c}$ is both the frame of tensor ideals and the  frame of Thomason sets of $\spc \T^c$, i.e.\ the frame associated to the Hochster dual $\hspc \T^c$.   Also recall from Definition \ref{skuladef}  that $\wf_{\hspc \T^c}\colon \skula{\hspc \T^c }\to \hspc \T^c$ is the map induced by the identity function.  Recall also that by definition $\skula{\hspc \T^c }=\lspc \T$.

\begin{theorem}\label{main4}
Let $\T$ be a rigidly compactly generated tensor triangulated category.  
  
\begin{enumerate}
\item\label{main42} Moreover,  the assignment $\wV\mapsto \Ac{\Ga_\wV 1}$ defines a  frame homomorphism
\[\ga_\T\colon  \TT{\T^c}\to \Ab{\T}.\]
\item\label{main43} The following are equivalent.
	\begin{enumerate}[(a)]
	\item\label{main43a} The category $\T$ is supportive.  
	\item\label{main43b} For any collections of Thomason subsets $\{\wV_i\},\{\wU_i\},\{\wV_j\},\{\wU_j\}\sbe\nobreak \TT{\T^c}$, if the sets
	\[\bigcup_{i} \wV_i\cap\wco{\wU_i}=\bigcup_{j} \wV_j\cap\wco{\wU_j}\]
are equal, then the following Bousfield classes coincide
	\[\Ac{\coprod_i \Ga_{\wV_i}L_{\wU_i} 1}=\Ac{\coprod_j \Ga_{\wV_j}L_{\wU_j} 1}.\]
	\item\label{main43c} There exists a frame homomorphism $\eta$ which completes the following commutative diagram.
	\[\xymatrix{
	\TT{\T^c}  \ar[d]_{\bF\wf_{\hspc \T^c}}\ar[r]^{\ga_\T} & \Ab{\T}  \\
	\Fr{\lspc\T} \ar@{.>}[ur]^\eta & 	\\
	}\]
	\end{enumerate}
\end{enumerate}
\end{theorem}

\begin{corollary}\label{main4cor}
A rigidly compactly generated tensor triangulated category $\T$ is supportive if the idempotent Bousfield frame $\Ab{\T}$ is spatial.  

\end{corollary}

\begin{example}\label{Lebesgue}
Unfortunately, the idempotent Bousfield frame $\Ab{\T}$ is not always spatial.  Let $\bM$ be the  Lebesgue measurable sets of the real numbers.  Write $A\sim B$ for $A,B\in \bM$ if their symmetric difference has measure 0.  The Boolean algebra $\bB=\bM/\sim$ has no points and thus is a nonspatial frame.  In \cite{Stevenson17a}, Stevenson constructs  a triangulated category $\T$ satisfying our usual assumptions such that $\Ab{\T}=\bB$.  
\end{example}

Before we prove the corollary, we recall an easy lemma. This was observed in the stable homotopy category by Hovey and Palmieri in  \cite[Proposition 4.5]{HoveyPalmieri98}.

\begin{lemma}\label{main41} 
For a Thomason subset $\wV\in\TT{\T^c}$,  the Bousfield classes $\Ac{\Ga_\wV 1}$ and $\Ac{L_\wV 1}$ are complements in $\Ab{\T}$, i.e.
\[\Ac{\Ga_\wV 1}\jn \Ac{L_\wV 1}=\Ac{1}\quad\quad\quad\Ac{\Ga_\wV 1}\mt\Ac{L_\wV 1}=\Ac{0}.\]  
\end{lemma}

\begin{proof}  First, note $\Ac{\Ga_\wV 1}\mt \Ac{L_\wV 1}=\Ac{\Ga_\wV 1\tns L_\wV 1}=\Ac{0}$.
Next, if $T$ is in 
\[ \Ac{\Ga_\wV 1}\jn \Ac{L_\wV 1}=\Ac{\Ga_\wV 1}\cap \Ac{L_\wV 1}\]
then the idempotent triangle $\Ga_\wV T\to T\to L_\wV T\to $
implies that $T=0$.  Hence this intersection is $\{0\}=\Ac{1}$. 
\end{proof}

\begin{proof}[{Proof of Corollary \ref{main4cor}}]

Apply the functor $\spc$ to the diagram in Theorem \ref{main4} (\ref{main43}) (\ref{main43c}).  
\[\xymatrix{
	\hspc\T^c  & \spc \Ab{\T}  \ar[l]_{\spc \ga_\T}\\
	\skula{\hspc \T^c } \ar[u]^{\wf_{\hspc\T^c}} & 	\\
}\]
Since $\Ab{\T}$ is spatial, $\Ab{\T}$ and the frame of open sets $\bF\spc \Ab{\T}$ are isomorphic.  Thus to every Bousfield class $\Ac{T}$, we associate the open set $\wD(\Ac{T})$.  

By  Lemma \ref{main41},  the Bousfield classes $\Ac{\Ga_\wV 1}$ and $\Ac{L_\wV 1}$ are complements.  Thus their corresponding open sets $\wD(\Ac{\Ga_\wV 1})$ and $\wD(\Ac{L_\wV 1})$ are set theoretic complements.  In particular, the closed set $\wco{\wD(\Ac{\Ga_\wV 1})}=\wD(\Ac{L_\wV 1})$ is open. We conclude that 
\[{\left(\spc \ga_\T\right)}^{-1}(\wco{\wV})=\left({\left(\spc \ga_\T\right)}^{-1}(\wV)\right)^\wc=\wco{\wD(\Ac{\Ga_\wV 1})}=\wD(\Ac{L_\wV 1})\]
is open, and thus $\spc \ga_\T$ is not only continuous with respect to the Hochster dual  topology, $\hspc \T^c$, but with respect to the localising topology  $\skula{\hspc \T^c }$.  

Returning to the above diagram,  we have actually  just shown that there exists a continuous function $\wg\colon \spc \Ab{\T} \to \skula{\hspc \T^c }$ making the above diagram commute.  Since $\Ab{\T}$ is spatial, $\bF\spc\Ab{\T}\cong \Ab{\T}$ by Theorem \ref{framecats}.  Thus $\eta=\bF \wg$ satisfies the hypotheses of Theorem \ref{main4} (\ref{main43}) (\ref{main43c}).
\end{proof}

\begin{remark}\label{subframe}

The proof of Corollary \ref{main4cor} still holds if we only assume there exists a sub lattice $\bX\sbe \Ab{\T}$ such that
\begin{enumerate}[(a)]
\item $\bX$ is a spatial frame
\item the inclusion map $\bX\hookrightarrow \Ab{\T}$ is a frame homomorphism
\item for any Thomason set $\wV\sbe \spc\T^c$, the Bousfield classes $\Ac{\Ga_{\wV}1}$ and $\Ac{L_\wV 1}$ are in $\wX$.  
\end{enumerate}
We discuss the utility of this remark in Section \ref{Q3}.
\end{remark}

The rest of the section is devoted to proving Theorem \ref{main4}.

\begin{proof}[{Proof of Theorem \ref{main4} (\ref{main42})}]

The frame homomorphism $\ga_\T$ preserves extrema since $\ga_\T (\spc \T)=\Ac{1}$ and $\ga_\T(\emptyset)=\Ac{0}$.  The frame homomorphism $\ga_\T$ preserves finite meets since any Thomason subsets $\wV,\wU\in \TT{\T^c}$ satisfy
\[\Ac{\Ga_\wV 1}\mt \Ac{\Ga_\wU 1}=\Ac{\Ga_\wV 1\tns \Ga_\wU 1}=\Ac{\Ga_\wV \Ga_\wU 1}=\Ac{\Ga_{\wV\cap \wU} 1}\]
by Theorem \ref{idempotenttriangle}.

We now check that $\ga_\T$ preserves arbitrary joins.  Let $\{\wV_i\}\in \TT{\T^c}$ be a family of Thomason subsets, and set $\wV=\bigcup \wV_i$.  Since
\[\Ac{\Ga_\wV 1}\sbe \bigcap_i \Ac{\Ga_{\wV_i}1}=\bigjn_i \Ac{\Ga_{\wV_i}1}\]
we show the reverse containment. 

Consider the following computation where $\thick$, $\thick^\tns$, and $\thick^{\sqrt{\tns}}$ respectively denote the thick closure, thick tensor ideal closure, and the thick radical tensor ideal closure;
\begin{equation}\tag{$\star$} 
\thick\bigcup {\T_{\wV_i}}^c=\thick^{\tns} \bigcup {\T_{\wV_i}}^c=\thick^{\sqrt{\tns}} \bigcup {\T_{\wV_i}}^c={\T_{\wV}}^c.
\end{equation}
 The first equality holds because the thick closure of two tensor ideals is again a tensor ideal.   The second follows from rigidity; see  Lemma \ref{rigid}.  The last equality is Theorem \ref{class} (2).  

 Suppose $T$ is an object in $\bigjn_i \Ac{\Ga_{\wV_i}1}$, i.e. suppose $\Ga_{\wV_i} T=0$ for all $i$.  Then $\hm{\T}{{\T_{\wV_i}}^c}{T}=0$ for all $i$ by Theorem \ref{idempotenttriangle} (\ref{acyc}).  Equation ($\star$) implies that   $\hm{\T}{{\T_{\wV}}^c}{T}=0$.  Hence  $T\in \Ac{\Ga_\wV 1}$ by Theorem \ref{idempotenttriangle} (\ref{acyc}).  
\end{proof}

\begin{proof}[{Proof of Theorem \ref{main4} (\ref{main43}), (\ref{main43a})$\implies$(\ref{main43b})}]

 Assume $\T$ is supportive.   Suppose that  $\{\wV_i\},\{\wU_i\},\{\wV_j\},\{\wU_j\}\sbe \TT{\T^c}$ are collections of Thomason subsets, and that
\[\bigcup_{i} \wV_i\cap\wco{\wU_i}=\bigcup_{j} \wV_j\cap\wco{\wU_j}.\]
 Let $T\in \Ac{\coprod_i \Ga_{\wV_i}L_{\wU_i} 1}$. For all $i$ 
 \[\Ga_{\wV_i}L_{\wU_i} 1\tns T=\Ga_{\wV_i}L_{\wU_i}T=0.\]
 By Theorem \ref{main2} (\ref{main1c'}), this means that 
 \[\supp T\cap\left(\bigcup_i \wV_i\cap\wco{\wU_i}\right)=\emptyset.\]
   But our assumption then implies $\supp T\cap\wV_j\cap\wco{\wU_j}=\emptyset$ for all $j$, and so
 \[\Ga_{\wV_j}L_{\wU_j} 1\tns T=\Ga_{\wV_j}L_{\wU_j} T=0\]
by Theorem \ref{main2} (\ref{main1c'}). Therefore $T$ is  in the class $\Ac{\Ga_{\wV_j}L_{\wU_j} 1}$ for all $j$.  

Thus far we have shown that $\Ac{\coprod_i \Ga_{\wV_i}L_{\wU_i} 1}\sbe\Ac{\Ga_{\wV_j}L_{\wU_j} 1}$ for each $j$.  We conclude that
\[\Ac{\coprod_i \Ga_{\wV_i}L_{\wU_i} 1}\sbe\bigcap_j\Ac{ \Ga_{\wV_j}L_{\wU_j} 1}= \Ac{\coprod_j \Ga_{\wV_j}L_{\wU_j} 1}.\]
Equality follows from symmetry.   
\end{proof}

\begin{proof}[{Proof of Theorem \ref{main4} (\ref{main43}), (\ref{main43b})$\implies$(\ref{main43c})}]

Assume (\ref{main43b}).  	Define the function 
\[\eta\colon \Fr{ \skula{\hspc\T^c}}\to \Ab{\T}\quad\quad\mbox{by}\quad\quad\bigcup_{i} \wV_i\cap\wco{\wU_i}\mapsto \Ac{\coprod_i \Ga_{\wV_i}L_{\wU_i} 1}.\]
This function is well defined by  (\ref{main43b}).   Clearly, $\eta$ makes the diagram in Theorem \ref{main4} commute.

We check that $\eta$ is a frame homomorphism.   First $\eta$ preserves extrema.
 \[\eta(\spc \T^c)=\Ac{\Ga_{\spc \T^c} 1}=\Ac{1}\quad\quad\quad \eta(\emptyset)=\Ac{\Ga_{\emptyset} 1}=\Ac{0}\]
Next, we check that $\eta$ preserves finite meets.  The fourth equality is Theorem \ref{idempotenttriangle} (\ref{Intersections}). 
\begin{align*}
\eta\left(\bigcup_{i} \wV_i\cap\wco{\wU_i}\right)\mt\eta\left(\bigcup_{j} \wV_j\cap\wco{\wU_j}\right)
	&=\Ac{\coprod_i \Ga_{\wV_i}L_{\wU_i} 1}\mt \Ac{\coprod_j \Ga_{\wV_j}L_{\wU_j} 1}\\
	&=\Ac{\left(\coprod_i \Ga_{\wV_i}L_{\wU_i} 1\right)\tns\left(\coprod_j \Ga_{\wV_j}L_{\wU_j} 1\right)}\\
	&=\Ac{\coprod_i \Ga_{\wV_i}L_{\wU_i} 1 \tns \Ga_{\wV_j}L_{\wU_j} 1} \\
	&=\Ac{\coprod_i \Ga_{\wV_i\cap\wV_j}L_{\wU_i\cup \wU_j} 1} \\
	&=\eta\left(\coprod_i (\wV_i\cap\wV_j)\cap\wco{(\wU_i\cup \wU_j)}  \right) \\
	&=\eta\left(\left(\bigcup_{i} \wV_i\cap\wco{\wU_i}\right)\cap\left(\bigcup_{j} \wV_j\cap\wco{\wU_j}\right)\right)\\
\end{align*}

A similar calculation shows that $\eta$ preserves arbitrary joins.   
\end{proof}

\begin{proof}[{Proof of Theorem \ref{main4} (\ref{main43}), (\ref{main43c})$\implies$(\ref{main43a})}]

Assume (\ref{main43c}).  Every Thomason subset $\wV\in \TT{\T^c}$ satisfies $\eta(\wV)=\ga_\T(\wV)=\Ac{\Ga_\wV1}$.  And so $\eta(\wco{\wV})=\Ac{L_\wV1}$, by Lemma \ref{main41}.  Thus $\eta$ is the function defined in (\ref{main43b}).

 Suppose $\supp T=\emptyset$.  Then for every $\p\in \spc\T^c$ there exists a pair of Thomason subsets $\wV_\p,\wU_\p\in \TT{\T^c}$ with $\p\in \wV_\p\cap\wco{\wU_\p}$ such that
 \[\Ga_{\wV_\p}L_{\wU_\p} 1\tns T=\Ga_{\wV_\p}L_{\wU_\p} T=0.\]
 It follows that
 \begin{align*}
 T\in \bigcap_{\p\in\spc\T^c} \Ac{\Ga_{\wV_\p}L_{\wU_\p}1}=& \Ac{\coprod_{\p\in \spc\T^c}\Ga_{\wV_\p}L_{\wU_\p}1}\\
 &=\eta\left(\bigcup_{\p\in \spc\T^c}\wV_\p\cap\wco{\wU_\p} \right)=\eta(\spc\T^c)=\Ac{1}=\{0\}. \\
 \end{align*}
Therefore $T=0$.  
\end{proof}

\section{Topology and support}\label{assemblystuff}


\subsection{The assembly of a frame}

\begin{definition}\hfill

\begin{enumerate}
\item An element $\bx$ of a frame $\bX$ is \textit{complemented} if there exists an element $\co{\bx}\in \bX$ such that $\bx\mt\co{\bx}=0$ and $\bx\jn\co{\bx}=1$.  We call $\co{\bx}$ the \textit{complement} of $\bx$.  
\item A frame homomorphism $\al\colon\bX\to\bY$ is \textit{complemented} if $\al(\bx)$ is complemented in $\bY$ for every $\bx\in \bX$. 
\item A \textit{nucleus} of a frame $\bX$ is a function $\nu\colon \bX\to \bX$ satisfying the following.
	\begin{enumerate}[(a)]
	\item $\bx\le \nu(\bx)$ for all $\bx\in \bX$
	\item $\bx\le \by$  implies $\nu(\bx)\le \nu(\by)$ for all $\bx,\by\in \bX$
	\item $\nu\nu(\bx)=\nu(\bx)$  for all $\bx\in \bX$
	\item $\nu(\bx\mt \by)=\nu(\bx)\mt\nu(\by)$ for all $\bx,\by\in \bX$
	\end{enumerate}
\item For a frame $\bX$, let $\Neu{\bX}$ denote the set of Nuclei.  We call $\Neu{\bX}$ the \textit{assembly} of $\bX$. The assembly is partially ordered in the following manner: for nuclei $\nu,\mu\in \Neu{\bX}$, we say that $\nu\le \mu$ if $\nu(\bx)\le \mu(\bx)$ for all $\bx\in \bX$.  
\item For every element of a frame $\bx\in \bX$, let $\nu_\bx\colon  \bX\to \bX$ be the nucleus defined by $\by\mapsto \by\mt\bx$.
\end{enumerate}

\end{definition}

The assembly has many different equivalent constructions; see \cite[Chapters II,IV,VI]{PicadoPultr12} for an overview.

\begin{theorem}\label{skula}
Let $\bX$ be a frame.
\begin{enumerate}
\item The assembly $\Neu{\bX}$ is a frame.  
\item The assignment $\bx\mapsto \nu_\bx$ defines a complemented frame homomorphism $\al_\bX\colon \bX\to \Neu{\bX}$.
\item For any complemented frame homomorphism $\ph\colon \bX\to \bY$, there is a unique frame homomorphism $\tilde{\ph}\colon\Neu{\bX}\to \bY$ making the following diagram commute.
\[\xymatrix{
\bX \ar[d]^{\al_\bX} \ar[r]^\ph & \bY \\
\Neu{\bX}\ar@{.>}[ur]_{\exists! \tilde{\ph} } & \\
}\]
\end{enumerate}
\end{theorem}

\begin{proof}

Statements (1) and (2)  are  \cite[II.2.5 Proposition]{Johnstone82} and \cite[II.2.6 Lemma]{Johnstone82} respectively.  Statement (3) follows from the argument in \cite[II.2.9 Corollary]{Johnstone82}; see the unpublished notes \cite[Theorem 5.4]{Simmons06b} for a complete proof.  The result also follows from \cite[IV.6.3.1,II.5.3]{PicadoPultr12}.
\end{proof}


Let $\wX$ be a topological space.  There is a strong connection between the assembly and the Skula topology on $\wX$.  Recall from Definition \ref{skuladef} that $f_\wX\colon  \skula{\wX}\to \wX$ is the continuous function induced by the identity.

\begin{corollary}\label{skulacor}


There is a unique frame homomorphism $\sm_\wX$ making the following diagram commute.  
\[\xymatrix{
\Fr{\wX} \ar[d]^{\al_{\Fr{\wX}}} \ar[r]^{\Fr{\wf_\wX}} & \Fr{ \skula{\wX}} \\
\Neu{\Fr{\wX}}\ar[ur]_{\sm_\wX} & \\
}\]
\end{corollary}

\begin{proof}

For every open set $\wV\sbe \wX$, the elements $\wV$ and $\wco{\wV}$ of the frame $\Fr{ \skula{\wX}} $ are complements.  Thus the result follows from Theorem \ref{skula} (3).  
\end{proof}

For a continuous map $\wg\colon \wX\to \wY$, the frame map $\Fr{\wg}$ is complemented if and only if for every open set $\wV\sbe \wY$ the preimage  $\wg^{-1}(\wV)$  is clopen in $\wX$.   Arguing as  in Corollary \ref{main4cor}, the function $\wf_\wX$ is universal amongst such ``complemented'' continuous functions whose domain is $\wX$. In this sense, the assembly of a frame is the pointless analogy of the Skula topology.  However, the universal property in Theorem \ref{skula} makes the assembly a much more versatile object, as we will see in the next section.

Before we apply our results to tensor triangulated categories, we discuss when the assembly and $\Fr{ \skula{\wX}}$ coincide.  

\begin{definition}\ 
\begin{enumerate}
\item For a space $\wX$  and a subspace $\wS\sbe \wX$, a point $\p\in \wS$ is  {\it weakly isolated} if there exists an open set $\wV\sbe \wX$ such that 
\[\p\in\wV\cap\wS\sbe \bar{\p}.\]
\item Let $\bX$ be a frame.  
	\begin{enumerate}[(i)]
	\item A prime $\bp\in \bX$ is \textit{minimal} over $\bx$ if it is minimal amongst primes containing $\bx$.  Let $\min(\bx)$ denote the the minimal primes of $\bx$.  
	\item A prime $\bp\in \min(\bx)$ is \textit{essential} over $\bx$ if 
	\[\bx=\bigmt_{\bq\in \min(\bx)} \bq\quad\quad\mbox{and} \quad\quad \bx\ne\bigmt_{\substack{\bq\in \min(\bx)\\ \bq\ne\bp}}\bq\]
	\end{enumerate}
\end{enumerate}
\end{definition}

\begin{theorem}\label{wscdef}
The following is equivalent for a topological space $\wX$.
\begin{enumerate}
\item\label{wscdefa} The map $\sm_\wX$ is an isomorphism of frames.  
\item\label{wscdefb} Every nonempty closed set has a weakly isolated point.
\item\label{wscdefc} Every closed set is the closure of its weakly isolated points.
\item\label{wscdefd} Every open set $\wV\sbne\wX$ has an essential prime in the frame $\Fr{\wX}$.
\item\label{wscdefe} Every element in the frame   $\Fr{\wX}$ is the meet of its essential primes.
\item\label{wscdefg}  For every surjective frame homomorphism $\Fr{\wX}\to \bY$, the frame $\bY$ has enough points.
\end{enumerate}
\end{theorem}

\begin{definition}
Any space satisfying the equivalent conditions above is called \textit{weakly scattered}.
\end{definition}

\begin{remark}\ 
\begin{enumerate}
\item When $\wX$ is $T_0$, it is easy to check that  (\ref{wscdefc}) is equivalent to the following: every nonempty closed set $\wU\sbe \wX$ is the closure of some discrete subspace $\wS\sbe \wX$.
\item When $\bX$ is sober, it is easy to check that (\ref{wscdefe})  equivalent to the following: every surjective frame homomorphism $\Fr{\wX}\to \bY$ is induced by a frame isomorphism $\Fr{\wS}\cong\bY$ for some subspace $\wS\sbe \bY$.  See \cite[VI.2.2.1]{PicadoPultr12}
\end{enumerate}
\end{remark}

\begin{proof}[{Proof of Theorem \ref{wscdef}}]

 In \cite[Theorem 3.4, Corollary 3.7]{NiefieldRosenthal87}, it is shown that (\ref{wscdefa}), (\ref{wscdefb}), (\ref{wscdefd}), (\ref{wscdefe}), and (\ref{wscdefg}) are equivalent.  In  \cite[Theorem 4.4]{Simmons80}, it is shown that (\ref{wscdefa}) and (\ref{wscdefc}) are equivalent. 

\end{proof}

\subsection{Support and the assembly}\label{wss}

We now arrive to the point of our pointless machinery.  Let $\T$ be a rigidly compactly generated tensor triangulated category. Recall the frame homomorphism $\ga_\T\colon \TT{\T^c}\to \Ab{\T}$ from Theorem \ref{main3}.  By Lemma \ref{main41}, $\ga_\T$ is complemented.  Therefore, Theorem \ref{skula} (3) implies that there is a unique frame homomorphism $\tilde{\ga}_\T$ making the following  diagram commute.
\[\xymatrix{
\TT{\T^c} \ar[d]_{\al_{\bT(\T^c)}} \ar[r]^{\ga_\T} & \Ab{\T} \\
\Neu{\TT{\T^c}}\ar@{.>}[ur]_{ \tilde{\ga}_\T}& \\
}\]
By Theorem \ref{main4},  $\T$ is supportive if and only if there exists a frame homomorphism $\eta$ making the following diagram commute.
\[\xymatrix{
	\TT{\T^c} \ar[dd]_{\al_{\bT(\T^c)}} \ar[dr]_{\bF\wf_{\hspc\T^c}}\ar[drrr]^{\ga_\T} & & & \\
						&  \Fr{ \lspc \T} \ar@{.>}[rr]^\eta && \Ab{\T} 	\\
	\Neu{\TT{\T^c}}\ar[ur]^{\sm_{\hspc\T^c}}\ar[urrr]_{ \tilde{\ga}_\T}& & &\\
}\]

%
%

By Theorem \ref{wscdef},$\sm_{\TT{\T^c}}$ is an isomorphism  if and only if   the Hochster  dual of $\spc \T^c$ is weakly scattered, motivating the following definition.

\begin{definition}
A spectral space $\wX$ is \textit{Hochster weakly scattered} if its Hochster dual $\wX^\da$ is weakly scattered.  By Example \ref{ZPThomason}, $\wX$ is Hochster weakly scattered if and only if for every Thomason subset $\wU\sbe\wX$, there exists a point $\p\notin \wU$ and a Thomason subset $\wV\sbe  \wX$ such that
\[\p\in \wV\cap \wco{\wU}\sbe \downarrow \p.\]
\end{definition}

 We have now proven the following result.  

\begin{theorem}\label{main5}

A  rigidly compactly generated tensor triangulated category $\T$ is supportive if $\spc \T^c$ is Hochster weakly scattered.
\end{theorem}

We now try to understand which spectral spaces are Hochster weakly scattered.  First, we consider the dual question: when is a spectral space weakly scattered?

\begin{definition}

Let $R$ be a commutative ring and $I\sbe R$ and ideal.  A prime $\p\in\spec R$ is an essential divisor of $I$ if $\p$ is an essential prime of $\sqrt{I}$ in the frame of radical ideals. This means that $\p\in \min I$ but
\[\sqrt{I}\ne \bigcap_{\substack{\p\in \min(I)\\ \p\ne \q }} \q\]

\end{definition}

\begin{lemma}
The following are equivalent for a commutative ring $R$.  
\begin{enumerate}
\item $\spec R$ is weakly scattered.
\item Every proper ideal has an essential prime divisor.
\item\label{asdf} For every proper radical ideal $I\sbe R$, there is a prime $\p$ such that 
\[(I:(I:\p))=\p.\]
\item Every radical ideal can be written as an irredundant intersection of primes.  
\end{enumerate}
In particular, a spectral space $\wX$ is Hochster weakly scattered if and only if\break $\wX^\da\cong \spec R$ for some commutative ring $R$ satisfying these hypotheses.  
\end{lemma}

\begin{proof}

By  \cite[Lemma 1]{Kirby69}, any  prime ideal satisfying the conclusion of (\ref{asdf}) is an essential prime divisor of $I$.  The rest of the result is \cite[Theorem 4.1]{NiefieldRosenthal87}.  

\end{proof}

The following examples demonstrate Hochster weakly spectral spaces are tricky. 

Even for a relatively nice non-Noetherian ring $R$, $\spec R$ need not be Hochster weakly scattered.

\begin{example}
Every Noetherian spectral space $\wX$ is weakly scattered.  Indeed, for such a space, every closed set $\wV\sbe\wX$ can be written as a union 
\[\wV=\wV_1\cup\cdots\cup \wV_n\]
  of irreducible closed sets $\wV_i$.  The generic point of each $\wV_i$ is weakly isolated in $\wV$.  

Consider the Hochster dual $\wY=\wX^\da$.   The space $\wY$ is Hochster weakly scattered, since $\wY^\da=\wX$.     Spaces of this form are rather bizarre.  For instance
\begin{itemize}
\item every specialisation closed subset in $\wY$ is closed
\item the set of  quasi-compact open sets of $\wY$ satisfies DCC
\item every quasi-compact open sets of $\wY$  is a finite union irreducible open sets.   
\end{itemize} 
\end{example}

\begin{example}\label{nonexample}

Let $R= k[x_1,x_2,\cdots]$ with $k$ algebraically closed.  Then $\spec R$ is not Hochster weakly scattered. 
\begin{proof}
We claim that ${(\spec R)}^\da$ itself has no weakly isolated points.  Let $\p\in \spec R$ and $\wV\sbe \spec R$ a Thomason subset with $\p\in \wV$. We claim that $\wV$ is not contained in $\bar{\p}=\downarrow \p$, i.e.\ we must produce a prime $\q\nsubseteq \p$ which is in $\wV$.  We may assume that $\wV=\wV(f_1,\dots,f_s)$ with $f_1,\dots,f_s\in R$.  Each $f_i$ is a polynomial in a finite set of variables, and so there is an $n$ such that each $f_i$ is in $S=k[x_1,\dots,x_n]$.  Let $\mm\sbe S$ be a maximal ideal containing the $f_i$.   By the Nullstellsatz, after a suitable coordinate change, we may write $\mm=(x_1,\dots,x_n)$.  Now either $x_{n+1}$ or $x_{n+1}+1$ is not in the ideal $\p$.  After again changing coordinates, we may assume $x_{n+1}\notin \p$.  The ideal  $\q=(x_1,\dots,x_n,x_{n+1})\sbe R$ is our desired prime.  
\end{proof}
\end{example}

\subsection{Hochster scattered spaces}

In order to relate our theory to the work of Balmer, Favi, and Stevenson, we discuss a special class of weakly scattered spaces. 

\begin{definition}
A space is {\it scattered} if every closed subset has an \textit{isolated} point which is a point which is open in the subspace topology.  A space is {\it Hochster scattered} if it is spectral and the Hochster dual is scattered.

\end{definition}

Scattered space first arose in connection with the following famous construction.  

\begin{con}[Cantor-Bendixon]
Let $\wX$ be a topological space. A point in a subspace $\p\in\wY\sbe  \wX$ is {\it isolated}  in $\wY$ if it is open in the subspace topology. Equivalently,  there is an open set $\wV\sbe \wX$ such that $\{\p\}=\wY\cap\wV$.  Note that isolated points are weakly isolated.  Let $\wI(\wY)$ denote the isolated points of $\wY$. Note that this is an open subspace of $\wY$.  

Inductively, for every ordinal $\al$, we will define an open subspace $\wX_{\le \al}$ and  a closed subspace $\wX_{>\al}$.  Set $X_{\le 0}=\emptyset$ and $\wX_{>0}=\wX$.  Assuming these sets are defined for $\al$, set $\wX_{\le \al +1}=\wX_{\le \al} \cup \wI(\wX_{>\al})$ and let  $\wX_{>\al+1}$ be its complement.  These spaces are open and closed respectively.  When $\al$ is a limit ordinal, set
\[\wX_{\le \al}=\bigcup_{\be<\al} \bX_{\le \be}\]
and let $\wX_{>\al+1}$ be its complement.  These spaces are open and closed respectively.

The {\it Cantor-Bendixon rank} of $\wX$ is then the smallest ordinal  $\al$ such that\break $\wX=\wX_{\le \al}$.  If no such ordinal exists, we say that  $\wX$ has no Cantor-Bendixon rank.
\end{con}

Recall the definition of a visible point from Subsection \ref{vis}.

\begin{lemma}\label{scat}
The following are equivalent for a $\sT_0$ space $\wX$.
\begin{enumerate}
\item $\wX$ is scattered
\item $\wX$ is weakly scattered and $\sT_{\frac{1}{2}}$
\item $\wX$ has Cantor-Bendixon rank
\item The assembly $\Neu{\Fr{\wX}}$ is a Boolean algebra
\end{enumerate}
If a space is Hochster scattered, then all of its points are visible. 
\end{lemma}

\begin{proof}

Statements (1) and (2) are equivalent by \cite[Proposition VI.8.1.1]{PicadoPultr12}.  The equivalence of (1) and (3) is an elementary exercise in topology.  Statements (1) and (4) are equivalent by \cite[Theorem 9]{Simmons78}.  By Lemma \ref{Noeth}, if the Hochster dual of a spectral space is $\sT_{\frac{1}{2}}$, then all the points are visible, proving the last statement.  
\end{proof}

The following examples of Hochster scattered spaces were inspired by Stevenson's work in \cite{Stevenson17}.  Recall that the constructible topology of a spectral space $\wX$ is generated by the Hochster dual topology and  the Zariski topology (the given topology on $\wX$).  

\begin{lemma}
A spectral space $\wX$ is  Hochster scattered if one of the following hold.
\begin{enumerate}
\item $\wX$ is Noetherian.
\item $\wX$ carries the constructible topology and has Cantor-Bendixon rank.
\end{enumerate}
\end{lemma}
\begin{proof}
The closed points of any Noetherian spectral space are isolated in the Hochster dual.  Since any closed subset of a Noetherian topological space is also a Noetherian spectral space, this proves (1). This is essentially the argument in \cite[Lemma 4.3]{Stevenson17}.

Assume (2).  Then any Hochster closed set $\wU\sbe \wX$ is Zariski closed. By Lemma \ref{scat} $\wX$ is scattered, so there exists a point  $\p$ such that $\{\p\}=\wV\cap\wU$ for some Zariski open set $\wV$.  We may assume that $\wV$ is quasi-compact open.  We claim that $\wco{\wV}$ is a quasi-compact clopen set, and thus $\wV$ is Hochster open.  The closed set  $\wco{\wV}$  is open in the constructible topology and thus also in the Zariski topology.  Thus $\wco{\wV}$ is clopen.  But clopen sets are automatically  quasi-compact, completing the proof.  
%
\end{proof}

Therefore if $\T$ is a rigidly compactly generated tensor triangulated category, $\T$ is supportive if $\spc\T^c$ is one of the above spaces. Stevenson proved this  in \cite{Stevenson17} assuming that $\T$ has a monoidal model. A close examination of Stevenson's paper implies a much stronger result.

\begin{theorem}\label{loctoglob}

Suppose $\T$ is a rigidly compactly generated tensor triangulated category and has monoidal model.  Suppose further that $\spc \T^c$ is Hochster scattered.  Then $\T$ satisfies the local-to-global principle, i.e. for every $T\in \T$ 
\[T\in \loc^\tns (\Ga_\p T\mid \p\in \supp T).\]

\end{theorem}

\begin{proof}

Let $\wX$ be the collection of all Hochster scattered spectral spaces.  For every $\wX\in \gX$, let $\dim_{\wX} \wX\to \mbox{Ord}$ denote the Cantor-Bendixon level computed in the Hochster dual.  These functions are well defined by Lemma \ref{scat}.  The collection $\cD=\{\dim_{\wX}\mid \wX\in \gX\}$ is what Stevenson calls a class of spectral dimension functions compatible with $\gX$; see \cite[Definition 3.5]{Stevenson17}.  In Theorem 3.7 and Lemma 3.13 of \cite{Stevenson17}, Stevenson shows that $\T$ satisfies the local-to-global principle if $\spc \T^c\in \gX$ and $\cD$ is closed under spectral subspaces.   However, an examination of the proof of Theorem 3.7 reveals that we only need to show that $\gX$ is closed under complements of Thomason subsets.

So we now check that $\gX$ has this property. Let $\wX\in \gX$ and take a Thomason subset $\wV\sbe \bX$.   Its complement, $\wco{\wV}$, is a spectral subspace and its Hochster dual is a subspace of $\wX^\da$.   Since subspaces of scattered spaces are scattered, ${(\wco{\wV})}^\da$ is scattered, and so $\wco{\wV}\in \gX$.
\end{proof}
%

\begin{example}
Let $R$ be an absolutely flat ring that is not semi-Artinian. In \cite[Theorem 4.7]{Stevenson14}, Stevenson shows that local-to-global principle fails in $D(R)$, and so $\spec R$ cannot be Hochster scattered by Theorem \ref{loctoglob}.  Furthermore, $\spec R$ is not Hochster weakly scattered: the Krull dimension of $R$ is 0, and so every prime is visible.  However, $D(R)$ is supportive either by Stevenson's arguments in loc.cit.\ or by Theorem \ref{main6}. Therefore $\T$ supportive does not imply $\spc \T^c$ is Hochster weakly scattered.
\end{example}

We summarise the following implications for a spectral space $\wX$
\[\mbox{Noetherian}\implies \mbox{Hochster scattered}\implies \mbox{all points are visible}.\]
Moreover, these implications are strict.

\begin{example}

Let $\wX=\bN\cup\infty$. Let $\wV_n=[n,\infty]$ be the nontrivial open sets of $\wX$. The space $\wX$ is spectral but not Noetherian.  The non-trivial Hochster open sets are of the form $[1,n]$ with $n\in \NN$.  In particular, $1$ is an open point in $\wX$, and $n+1$ is isolated in the Hochster closed set $[n+1,\infty]$.  Thus $\wX$ is Hochster scattered.  

\end{example}

The following example was communicated to the author by Bill Fleissner.  

\begin{example}\label{cantorexam}

Let $\wC$ denote the Cantor set.  This space is spectral,  and so let $\wX=\wC^\da$.   Now $\wX^\da=\wC$, and so the Hochster dual of $\wX$ is $\sT_{\frac{1}{2}}$.  Thus every point is visible by Lemma \ref{Noeth}.  However, $\wX$ is not Hochster scattered.  

\end{example}

\section{Questions}\label{questionstuff}

\subsection{Supportiveness}

Let $\T$ be a rigidly compactly generated tensor triangulated category. This article raises the following question.

\begin{question}\label{Q1}
Is $\T$ always supportive? 
\end{question}

By our work in Section \ref{framestuff},  supportiveness is a restriction on the idempotent Bousfield  lattice.  Suppose for instance, that $\spc\T^c$ is not Hochster weakly scattered and $\T$ is supportive.  
Then $\Ab{\T}$ cannot be isomorphic to the assembly $\Neu{\TT{\T^c}}$ by Theorem \ref{main4}, and so $\Ab{\T}$ cannot be any old frame.  This may not seem very strong, but the  Bousfield lattice is notoriously ill-behaved, see for example \cite{Neeman00}, \cite{DwyerPalmieri08}, \cite[Section 4]{Stevenson14}, and\cite{Stevenson17a}.  Thus a negative answer to Question \ref{Q1} is likely.

On the other hand, Theorem \ref{intro3} describes three disparate cases where supportiveness hold, suggesting that supportiveness might be common.     

\begin{example}

As we learned in Example \ref{cantorexam}, there exists a commutative ring $R$ such that ${(\spec R)}^\da$ is the Cantor set.  Thus $\spec R$ is not Hochster weakly scattered.  However, since the Cantor set has Krull dimension 0, so does $\spec R$.  Thus $\DD{R}$ is supportive by Theorem \ref{main6}.  This example generalises to any dimension 0 that is not Hochster weakly scattered.   

\end{example}

This example shows that for commutative rings, $\DD{R}$ is supportive  not for topological reasons but for algebraic reasons.  

\begin{question}

Is $\DD{R}$ supportive for all commutative rings $R$?
\end{question}

%

\subsection{Smashing subcategories}

Assume as usual that $\T$ is  rigidly compactly generated tensor triangulated category, but now assume that $\T$ has a monoidal model.   Let $\bS$ be the lattice of smashing subcategories.  Recall that for every $\cS\in \bS$, there are coproduct preserving local cohomology and localisation functors $\Ga_\cS$ and $L_\cS$ and an idempotent triangle 
\[\Ga_{\cS} 1\to 1\to L_{\cS} 1\to .\]
See \cite{BalmerFavi11}.  In \cite{BKS17}, the authors prove that $\bS$ is a frame, and moreover, by \cite[Theorem 3.5]{BalmerFavi11} the assignment $\cS\mapsto \Ac{\Ga_\cS 1}$ defines a frame homomorphism $\ph\colon\bS\to \Ab{\T}$.  This homomorphism is complemented since $\Ac{\Ga_\cS1}$ and $\Ac{L_\cS1}$ are complements in $\Ab{\T}$.  

Therefore we can apply the results of Theorem \ref{skula} concerning the assembly and and the analysis of Section \ref{wss} extends to this general case.  In particular we have the following commutative diagram of frame homomorphisms.  
\[\xymatrix{
			& \TT{\T^c} \ar[dl]_{\bF\wf_{\hspc \T^c}}  \ar[d]^{\al_{\bT(\T^c)}} \ar@{^{(}->}[r]	& \bS  \ar[d]^{\al_{\bS}} \ar[r]^\ph	&\Ab{\T} \\
\Fr{\lspc \T}	&\Neu{\TT{\T^c}} \ar[l]^{\sm_{\hspc \T^c}} \ar[r] 							& \Neu{\bS} \ar[ur]_{\tilde{\ph}} \\
}\]

Now suppose for a moment that $\bS$ is a spatial frame.  Then the open sets of $\spc \bS$ are in bijection with $\bS$.  For every $\cS\in \bS$, let $\wD(\cS)$ be the associated open set.  

\begin{definition}

Let $T\in \T$. A point $\p\in \spc \bS$ is in $\smsupp T$ if for every $\cS,\cS'\in \bS$ with $\p\in \wD(\cS)\cap\wco{\wD(\cS')}$, we have
\[\Ga_{\cS} L_{\cS'} T\ne 0.\]
We say that $\T$ is smashing supportive if $\smsupp T=\emptyset$ if and only if $T=0$,  
\end{definition}

The main results of this paper, Theorems \ref{main1}, \ref{main2},  \ref{main3},  \ref{main4}, and  \ref{main5} all generalise to the smashing context, motivating the following.

\begin{question}

Is the smashing frame $\bS$ spatial?  When is $\T$ smashing supportive?
\end{question}

\subsection{Subframes of the Bousfield lattice}\label{Q3}

In Remark \ref{subframe} we observed that $\T$ is supportive  if the image of $\ga_\T\colon \:\TT{\T^c}\to \Ab{\T}$ lies in a spatial frame containing $\Ac{L_\wV 1}$ for all $\wV\in \TT{\T^c}$.  When can we apply this remark?   Recall from Example \ref{Lebesgue} that not all Boolean algebras are spatial.

\begin{question}\label{Q3q}

Suppose $\bB$ is a Boolean algebra and $\bF\sbe \bB$ is a spatial frame such that the inclusion map is a frame homomorphism.  Let $\bar{\bF}\sbe \bB$ be the smallest subset of $\bB$ closed under arbitrary joins, finite meets, and complements.   Is the complete Boolean algebra $\bar{\bF}$ always atomic?

\end{question}
 
Suppose Question \ref{Q3q} has an affirmative answer.  Then $\T$ is supportive whenever $\Ab{\T}$ is Boolean: let us explain. Since $\ga_\T$ is injective, $\im{\ga_\T}$ is a subframe.  Thus, an affirmative answer to the question implies that $\overline{\im{\ga_\T}}$ is spatial satisfies the hypotheses of Remark \ref{subframe}.  Hence $\T$ is supportive.  

Occasionally, $\Ab{\T}$ is indeed Boolean.  For instance every Boolean algebra is $\Ab{\T}$ for some $\T$ satisfying our usual assumptions by \cite{Stevenson17a}.  

\begin{lemma}
If $\T$ has no nilpotent elements, i.e.\ $T\tns T=0$ implies $T=0$, then every Bousfield class is idempotent and $\Ab{\T}$ is Boolean.  
\end{lemma}

\begin{proof}
First, if $\Ac{T}\ne \Ac{T\tns T}$, then there exists an $X\in \T$ such that $T\tns X\ne 0$ but $T\tns T\tns X=0$.  It follows that $T\tns X$ is nilpotent. We now prove the second statement. 

For a class $\cX\sbe \T$, set
\[\cX^\perp=\{T\in \T\mid T\tns X=0\ \forall X\in\cX\}.\]
Standard formalism implies that 
\[\cX\sbe \cX^{\perp\perp}\quad\quad\mbox{and}\quad\quad \cX\sbe \cY\implies \cY^\perp\sbe \cX^\perp\]
allowing us to conclude 
\[\cX^\perp=\cX^{\perp\perp\perp}.\]
By definition, for a single object $T\in \T$, 
\[T^\perp=\Ac{T}.\]

Let $\Ac{T}^*$ be the pseudo-complement of an element $\Ac{T}\in \Ab{\T}$.  We claim that $\Ac{T}^{**}=\Ac{T}$.  Since $\Ac{X}\in \Ab{\T}$ for every $X\in \T$, we compute
\[\Ac{T}^*=\bigjn_{\substack{\Ac{S}\in \Ab{\T} \\ \Ac{S}\mt \Ac{T}=\Ac{0}}} \Ac{S}=\bigjn_{\substack{S\in\T \\ S\tns T=0}} \Ac{S}=\bigcap_{X\in T^\perp} S^\perp=T^{\perp\perp}=\Ac{T}^\perp.\]
We now have
\[\Ac{T}^{**}={\Ac{T}^*}^\perp=T^{\perp\perp\perp}=T^\perp=\Ac{T}\]
where the second equality is the second to last equality of the previous calculation.  
\end{proof}

\subsection{Localising tensor ideals}

In Theorem \ref{main1} and Theorem \ref{main2}, we saw that $\supp$ gives the following maps.
\[\left\{\substack{ \mbox{Localising} \\ \mbox{tensor ideals} \\ \mbox{of } \T}\right\}
\substack{\xrightarrow{\supp} \\ \xleftarrow{\supp^{-1}}}
\left\{\substack{ \mbox{Closed } \\ \mbox{subsets} \\ \mbox{of } \lspc \T}\right\}\]

\begin{question}\label{bijideal}
When are the localising tensor ideals in bijection with the closed subsets of $\lspc \T$?
\end{question}

The appearance of closed subsets is surprising.  In Theorem \ref{class}, we saw that the radical tensor ideals of $\T^c$ are in bijection with the open sets, and thus form a frame.  However, if the closed subsets of $\lspc \T$ classify the localising tensor ideals, then the localising tensor ideals form a \textit{co-frame}, a lattice whose opposite lattice is a frame.  This motivates the following question.
\begin{question}\label{coframe}
Is the lattice of localising tensor ideals a co-frame?  Equivalently, is the lattice of localising tensor ideals a frame when ordered by reverse inclusion?
\end{question}

This question has a positive answer when every localising tensor ideal is an idempotent Bousfield class, in which case $\Ab{\T}$ is the lattice in question.

It is not known whether or not the collection of localising tensor ideals is a set.  But  Question \ref{coframe} is independent of such foundational concerns, since frames and coframes need not be sets.  

\section*{Acknowledgements}

We thank John Reynolds for many conversations about point set topology.  We also thank Bill Fleissner for his examples concerning scattered spaces and his interest in this work.  We are also deeply indebted to Greg Stevenson, whose comments and careful reading of an early draft  greatly improved this article.

\bibliographystyle{amsplain}
\bibliography{/Users/williats/aamath/communications/BibliographyII}

\end{document}